\documentclass[11pt]{article}
\usepackage[left=3cm,right=3cm,top=2.2cm,bottom=2.2cm]{geometry}
\usepackage{amsmath,amssymb,mathrsfs,stmaryrd}
\usepackage{amsthm}
\usepackage{mathtools,bm}
\usepackage{dsfont}
\usepackage[hidelinks]{hyperref}
\usepackage[nameinlink,capitalize]{cleveref}
\usepackage{xcolor}
\usepackage[shortlabels] {enumitem}
\usepackage{microtype}
\usepackage[abbrev,msc-links,alphabetic]{amsrefs}
\usepackage{todonotes}
\usepackage{authblk}
\usepackage{titlefoot}
\usepackage{float}
\theoremstyle{plain}
\newtheorem{theorem}{Theorem}[section]

\newtheorem{corollary}[theorem]{Corollary}

\newtheorem{lemma}[theorem]{Lemma}

\newtheorem{assumption}[theorem]{Assumption}
\theoremstyle{definition}
\newtheorem{definition}[theorem]{Definition}

\newtheorem{example}[theorem]{Example}

\newtheorem{remark}[theorem]{Remark}

\newcommand{\E}{\ensuremath{{\mathbb E}}}

\usepackage{graphicx}
\begin{document}

\title{An application of the multiplicative Sewing Lemma to the high order weak approximation of stochastic differential equations}
\date{June 21, 2022}

\author{Antoine Hocquet}
\author{Alexander Vogler}

\affil{\small Technische Universit\"at Berlin, Berlin, Germany}

\maketitle

\unmarkedfntext{\textit{Mathematics Subject Classification (2020) ---} 60F05, 60G15, 60J35, 60H10, 60G53}

\unmarkedfntext{\textit{Keywords and phrases --- weak approximation, sewing lemma, cubature on Wiener space} }

\unmarkedfntext{\textit{Mail}: \textbullet$\,$ antoine.hocquet86@gmail.com $\,$\textbullet$\,$ vogler@math.tu-berlin.de}


\begin{abstract}
	We introduce a variant of the multiplicative Sewing Lemma in [Gerasimovi{\v{c}}s, Hocquet, Nilssen; J.\ Funct.\ Anal.\ 281 (2021)] which yields arbitrary high order weak approximations to stochastic differential equations, extending the cubature approximation on Wiener space introduced by Lyons and Victoir. Our analysis allows to derive stability estimates and explicit weak convergence rates. As a particular example, a cubature approximation for stochastic differential equations driven by continuous Gaussian martingales is given. 
\end{abstract}

\tableofcontents

\section{Introduction}

In this paper we derive an arbitrary high order weak approximation result for stochastic differential equations as an application of a version of the multiplicative sewing lemma, first introduced in \cite{GHN21}. This allows us to deduce weak convergence with rather weak assumptions in terms of the expected signature of the approximating driving signal. Our approximation covers in particular the cubature approximations on Wiener space introduced in \cite{article}. \\
The expected signature of a stochastic process provides significant information about its distribution. In \cite{article} this was used to construct a cubature approximation on Wiener space for the stochastic differential equation
\begin{align*}
X_t^x=x+\sum_{i=0}^{d}\int_{0}^{t}V_i(X_s^x)\circ dB_s^i,
\end{align*}
where $V_i\in \mathcal{C}_b^\infty(\mathbb{R}^{d},\mathbb{R}^d)$ and $B=(B^1,...,B^d)$  is  a $d$-dimensional Brownian motion and $B^0(t)=t$. One can view the solution $X_T$ as a functional of the Brownian path $(B_s)_{s\in [0,T]}$ 
\begin{align*}
X_T^x=\phi_{x,T}((B_s)_{s\in [0,T]}),\quad \phi_{x,T}:\mathcal{C}_0([0,T],\mathbb{R}^d)\rightarrow \mathbb{R}^d
\end{align*} 
and therefore write the expectation $P_{0,T}f(x):=\E\left[f(X_T^x)\right]$ as the integral over the Wiener space with respect to the Wiener measure $\mathbb{P}_W$,
\begin{align*}
P_{0,T}f(x)=\int_{\mathcal{C}_0([0,T],\mathbb{R}^d)}^{}f(\phi_{x,T}(\omega))\mathbb{P}_W(d\omega).
\end{align*}
The main idea of cubature on Wiener space now is to approximate such integrals by replacing the Wiener measure with some discrete cubature measure $\mathbb{Q}_{0,T}$, which is given by the law of a random path of bounded variation $W$, such that the expected signature of the Brownian motion and that of the random path $W$ coincide up to a certain order. Explicit examples for cubature measures can be found for example in \cite{article} and \cite{NV08}.\\
If the approximation is performed on arbitrary time intervals $[s,t]$, for $0\leq s\leq t\leq T$, this leads to a family of operators $Q_{s,t}f(x):=\E\left[f(X_t^{s,t,x})\right]$, where $X^{s,t,x}$ is the solution to 
\begin{align*}
X_r^{s,t,x}=x+\sum_{i=0}^{d}\int_{s}^{r}V_i(X_r^{s,t,x})dW_r^{s,t,i},\quad r\in [s,t].
\end{align*} 
Under suitable regularity assumptions it can then be shown, that
\begin{equation}\label{limit}
\lim_{n\rightarrow \infty}\prod_{i=0}^{k_n}Q_{t_{i}^n,t_{i+1}^{n}}f=P_{0,T}f,
\end{equation} 
for any sequence $\pi^n:=\{0=t_0^n<t_1^n<...<t_{k_n}^n=T\}$ with $|\pi^n|:= \max_{1\leq i\leq k_n} |t_{k_i}- t_{k_{i-1}}|\rightarrow 0$, as $n\rightarrow \infty$.\\

The multiplicative sewing lemma (Theorem \ref{sewing} below) happens to be suited to derive convergence results on semigroup approximations of the type (\ref{limit}), in a rather general setting. Aside from some stability estimate, it only requires $(Q_{s,t})_{(s,t)}$ to be almost multiplicative, i.e.\\ $|Q_{s,t}-Q_{s,u}Q_{u,t}|\leq \chi^z(s,t)$, for some control function $\chi$ and \( z>1\). For that it is sufficient that the expected signature of the noise approximation $(W^{(s,t)})_{(s,t)}$ be itself multiplicative up to a certain order, i.e.
\begin{align*}
	\E\left[S^m_{s,u}(W^{s,u})\right]\otimes\E\left[S^m_{u,t}(W^{u,t})\right]=\E\left[S^m_{s,t}(W^{s,t})\right],
\end{align*}
for some sufficiently large $m\in \mathbb{N}$, which will be one of our main assumptions. This condition is met for instance if the signature of the approximation process coincides with that of a square-integrable continuous martingale. Our approach therefore also covers cubature type approximations for stochastic differential equations driven by more general types of noise than Brownian motion.\\

The main tool to prove our main result will be the following version of the multiplicative sewing Lemma. A proof will be given in the Appendix.

\begin{theorem}\label{sewing}
	Let $(X,\|\cdot\|_X),(Y,\|\cdot \|_Y)$ be Banach spaces, such that  $X\subseteq  Y$ is dense with respect to $\|\cdot \|_Y$ and $\|x\|_Y\leq \|x\|_X$ for all $x\in X$. Let $\mu:\Delta_T\rightarrow L(X,X)\cap L(Y,Y)$, that is $\mu_{s,t}\in L(Y,Y)$ and $\mu_{s,t}|_X$ is continuous with respect to $\|\cdot \|_X$. \\
	Assume that $\mu_{s,s}=Id$ and there exist continuous functions $\epsilon_1,\epsilon_2:\Delta_T\rightarrow [0,\infty)$, which are increasing in the second component and decreasing in the first component, such that for any $(s,t)\in \Delta_T$ and any partition $\pi=\{s\leq t_0<t_1<...<t_k=t\}$ of $[s,t]$
	\begin{align*}
	\|\mu_{s,t}^\pi\|_{L(X,X)}\leq \epsilon_1(s,t),\quad \|\mu_{s,t}^\pi\|_{L(Y,Y)}\leq \epsilon_2(s,t),
	\end{align*}
	where 
	\begin{align*}
		\mu_{s,t}^\pi:=\mu_{s,t_1}\mu_{t_1,t_2}...\mu_{t_{k-1},t}.
	\end{align*}
	If there exists a control $\chi$, such that $(\mu_{s,t})_{(s,t)\in \Delta_T}$ is almost multiplicative in the sense that for all $s\leq u\leq t$ in $[0,T]$ 
	\begin{align*}
		\|\mu_{s,t}-\mu_{s,u}\mu_{u,t}\|_{L(X,Y)}\leq \chi^z(s,t),
	\end{align*}
	for some $z>1$, then there exists a unique evolution system of bounded linear operators $(\psi_{s,t})_{(s,t)\in \Delta_T}$ on $Y$, such that for all $(s,t)\in \Delta_T$
	\begin{align*}
		\|\psi_{s,t}-\mu_{s,t}^{\pi_n}\|_{L(X,Y)}\rightarrow 0,
	\end{align*}
	for any sequence of partitions $(\pi_n)_n$ with $|\pi_n|\rightarrow 0$. Furthermore it holds $\|\psi_{s,t}\|_{L(Y,Y)}\leq \epsilon_2(s,t)$,
	\begin{align*}
		\|\psi_{s,t}-\mu_{s,t}\|_{L(X,Y)}\leq 2^z\zeta(z)\epsilon^2(s,t)\chi(s,t)^z.
	\end{align*}
	and 
	\begin{align*}
		\|\mu_{s,t}^\pi-\psi_{s,t}\|_{L(X,Y)}&\leq  2^z\epsilon_1^2(s,t)\epsilon_2^2(s,t)\zeta(z)\chi(s,t)\max_{1\leq i\leq k}\{\chi^{z-1}(t_i,t_{i+1})\},
	\end{align*}
	where $\zeta$ is the Riemann zeta function. If in addition $(\mu_{s,t})_{(s,t)\in \Delta_T}$ is strongly continuous as an element of $L(Y,Y)$, i.e. for all $y\in Y$
	\begin{align*}
		\|\mu_{s,t}y-y\|_Y\rightarrow 0,
	\end{align*}
	as $(s,t)\rightarrow 0$, then $(\psi_{s,t})_{(s,t)\in \Delta_T}$ is also strongly continuous.
\end{theorem}


\section{Preliminaries}

Let us first introduce some useful notations and definitions, used throughout the whole manuscript. In subsection \ref{Assumptions} below we then introduce our assumptions on the noise approximations and state our main result. 

\subsection{Notation}

In the whole paper we consider an arbitrary but finite time horizon $T>0$ and fixed dimension $d\geq 1$. By $\Delta_T$ we denote the simplex
\begin{align*}
	\Delta_T:=\{(s,t)|0\leq s< t\leq T\}.
\end{align*}

Using a similar notation for multi-indices as in \cite{CM19}, we define
\begin{align*}
	\mathcal{A}:=\{\emptyset\}\cup \bigcup_{k\geq 1}\{0,1,...,d\}^k,\quad \mathcal{A}^1:=\{\emptyset\}\cup \bigcup_{k\geq 1}\{1,...,d\}^k
\end{align*}
and endow $\mathcal{A}$ with the concatenation operation 
\begin{align*}
	\alpha *\beta:=(\alpha_1,...,\alpha_k,\beta_1,...,\beta_l),
\end{align*}
for $\alpha=(\alpha_1,...,\alpha_k)\in \mathcal{A}$ and $\beta=(\beta_1,...,\beta_l)\in \mathcal{A}$. Let
\begin{align*}
	\mathcal{A}_l&:=\{\alpha\in \mathcal{A},\, \|\alpha\|\leq l\},\\
	\mathcal{A}_l^1&:=\{\alpha\in \mathcal{A}^1,\, \|\alpha\|\leq l\},
\end{align*}
where the $n$-tuples lengths are defined as
\begin{align*}
	|\alpha|:=\begin{cases}
		k & \text{if }\alpha=(\alpha_1,...,\alpha_k)\\
		0 & \text{if }\alpha=\emptyset
	\end{cases},
	\quad \|\alpha\|:=|\alpha|+|\{i\in \{1,...,d\}|\alpha_i=0\}|. 
\end{align*}
For smooth $V_i\in \mathcal{C}^\infty_b(\mathbb{R}^d,\mathbb{R}^d)$ regarded as vector fields on $\mathbb{R}^d$ by
\begin{align*}
	V_if(x)=\sum_{k=1}^{d}V_i^k(x)\frac{\partial}{\partial x_k}f(x),
\end{align*}
we define for $\alpha=(\alpha_1,...,\alpha_k)\in \mathcal{A}$
\begin{align*}
	V_\alpha f:=\begin{cases}
		f & \text{if }|\alpha|=0\\
		V_{\alpha_1}...V_{\alpha_k}f &\text{else.}
	\end{cases}
\end{align*}

With the notation introduced in \cite{CM19}, given $\alpha=(\alpha_1,...,\alpha_n)\in \mathcal{A}$ and $\omega\in \mathcal{C}^{1-var}([0,T],\mathbb{R}^d)$, we denote the iterated integral of the process $Y$ over the interval $[s,t]$ by
\begin{align*}
	I_{s,t}^\alpha[\omega](Y):=\int_{s}^{t}\int_{s}^{t_n}....\int_{s}^{t_2}Y_{t_1}d\omega_{t_1}^{\alpha_1}....d\omega_{t_n}^{\alpha_n},
\end{align*}
where we set $\omega^0_r=r$. For some continuous $d$-dimensional martingale $M$, we denote the iterated Stratonovich integral of the process $Y$ over the interval $[s,t]$ by
\begin{align*}
	I_{s,t}^\alpha[\circ M](Y):=\int_{s}^{t}\int_{s}^{t_n}....\int_{s}^{t_2}Y_{t_1}\circ dM_{t_1}^{\alpha_1}....\circ dM_{t_n}^{\alpha_n},
\end{align*}
where we set $M^0_r=r$. If $Y\equiv 1$, we omit the dependence on $Y$ and just write $I^\alpha_{s,t}[\omega]$ or $I^\alpha_{s,t}[\circ M]$ respectively.\\\\

Given a path $\omega\in \mathcal{C}^{1-var}([0,T],\mathbb{R}^d)$ and the basis $e_0,...,e_d$ of $\mathbb{R}\oplus \mathbb{R}^d$, where $e_i=(0,...,0,1,0,...,0)$, we define the series of iterated integrals over the interval $[s,t]$ with respect to $\omega$ by
\begin{align*}
	S_{s,t}[\omega]:=\sum_{k=0}^{\infty}\sum_{\alpha=(i_1,...,i_l)\in \mathcal{A}_k\setminus \mathcal{A}_{k-1}}I_{s,t}^{\alpha}[\omega]e_{i_1}\otimes ...\otimes e_{i_l}.
\end{align*}
Similarly, for a $d$-dimensional martingale $M$ we define 
\begin{align*}
	S_{s,t}[\circ M]:=\sum_{k=0}^{\infty}\sum_{\alpha=(i_1,...,i_l)\in \mathcal{A}_k\setminus \mathcal{A}_{k-1}}I_{s,t}^{\alpha}[\circ M]e_{i_1}\otimes ...\otimes e_{i_l}.
\end{align*}
For $m\in \mathbb{N}$, the truncated series of iterated integrals is then defined as $S_{s,t}^m[\omega]:=\pi_m(S_{s,t}[\omega])$ and $S_{s,t}^m[\circ M]:=\pi_m(S_{s,t}[\circ M])$ respectively.

\subsection{Settings and Assumptions}
\label{Assumptions} 

In the whole manuscript we consider a given family of random paths $(W^{s,t})_{(s,t)\in \Delta_T}$
\begin{align*}
	W^{s,t}:\Omega\rightarrow \mathcal{C}^{0,1-var}([s,t],\mathbb{R}^d),
\end{align*} 
defined on some probability space $(\Omega,\mathcal{F},\mathbb{P})$, which represent a given noise approximation on different time intervals. Here $\mathcal{C}^{0,1-var}([s,t],\mathbb{R}^d)$ denotes the set of absolutely continuous functions from $[s,t]$ to $\mathbb{R}^d$. Furthermore we let $W^{s,t,0}_r:=r$. When there is no confusion possible, we will write $I_{a,b}^{s,t,\alpha}$ and $S^{s,t,m}_{a,b}$ instead of $I_{a,b}^\alpha[W^{s,t}]$ and $S^m_{a,b}[W^{s,t}]$ respectively, for $(s,t)\in \Delta_T$ and $s\leq a\leq b\leq t$.

\begin{definition}
	A continuous map $\chi:\Delta_T\rightarrow [0,\infty)$ is called a control, if for all $s\leq u\leq t$ in $[0,T]$
	\begin{align*}
		\chi(s,u)+\chi(u,t)\leq \chi(s,t)
	\end{align*}
	and $\chi(s,s)=0$, for all $0\leq s\leq T$
\end{definition}

\begin{assumption}\label{a1}
	\begin{enumerate}[(i)]
		We assume that for all $s\le u\leq t$ in $[0,T]$, the processes $W^{s,u}$ and $W^{u,t}$ are independent, i.e.\ $\sigma(\{W^{s,u}_r|s\leq r\leq u\})$ and $\sigma(\{W^{u,t}_r|u\leq r\leq t\})$ are independent. Furthermore we assume that $W:=(W^{s,t})_{(s,t)\in \Delta_T}$ satisfies the following:
		\item There exists a control $\chi:\Delta_T\rightarrow [0,\infty)$ with $\chi(t,t+h)\in O(h)$, for any $t\in [0,T]$, such that $\mathbb{P}$ almost surely
		\begin{align*}
			\left|\dfrac{\partial}{\partial r}W_r^{s,t,i}\right|\leq C\dfrac{\chi(s,t)^{1/2}}{t-s},
		\end{align*}
		for almost every $r\in [0,T]$ and $i=1,...,d$.
		\item There exists $m\in \mathbb{N}$, $m\geq 2$, such that the following algebraic relation holds true: For every $s\leq u\leq t$ in $[0,T]$
		\begin{align*}
			&\E\left[S_{s,t}^{s,t,m}\right]=\E\left[S_{s,u}^{s,u,m}\right]\otimes \E\left[S_{u,t}^{u,t,m}\right].
		\end{align*}
		\item For any $(s,t)\in \Delta_T$ it holds for $i=1,...,d$
		\begin{align*}
			\E\left[W^{s,t,i}_t-W^{s,t,i}_s\right]=0.
		\end{align*}
		\item There exist a continuous function $A:[0,T]\rightarrow \mathbb{R}^{d\times d}, t\mapsto (a_{ij}(t))_{1\leq i,j\leq d}$, such that
		\begin{align*}
			\lim\limits_{h\downarrow 0}\frac{1}{h}\E\left[I_{t,t+h}^{t,t+h,(i,j)}\right]=a_{ij}(t),
		\end{align*}
		for all $t\in [0,T]$ and $1\leq i,j\leq d$.
	\end{enumerate}
\end{assumption}

\begin{remark}\label{cubature_remark}
	The second part of the assumption is always satisfied if the expected signature of $W^{s,t}$ coincides up to a certain order with the expected signature of a square-integrable continuous martingale $(M_t)_{t\in [0,T]}$ (possibly defined on some
	different probability space), such that for all $(s,t)\in \Delta_T$
	\begin{align*}
		\E\left[S_{s,t}^m[\circ M]\right]=\E\left[S_{s,t}^m[W^{s,t}]\right],
	\end{align*}
	for some $m\ge 2$. This can be seen by the following observation: Since $S_{u,t}^m[\circ M]$ is measurable with respect to $\sigma(M_r-M_u|r\geq u)$ and $S_{s,u}^m[\circ M]$ is measurable with respect to $\sigma(M_r|r\leq u)$, independence of the increments implies
	\begin{align*}
		\E\left[S_{s,u}^m[\circ M]\otimes S_{u,t}^m[\circ M]\right]=\E\left[S_{s,u}^m[\circ M]\right]\otimes \E\left[S_{u,t}^m[\circ M]\right].
	\end{align*}
	In a similar fashion as its done in \cite{article} for a process of bounded variation, one can show that for all $s\leq u\leq t$ in $[0,T]$ it holds almost surely
	\begin{align*}
		S_{s,u}[\circ M]\otimes S_{u,t}[\circ M]=S_{s,t}[\circ M],
	\end{align*}
	hence
	\begin{align*}
		\E\left[S_{s,t}^m[\circ M]\right]=\E\left[S_{s,u}^m[\circ M]\right]\otimes \E\left[S_{u,t}^m[\circ M]\right].
	\end{align*}
	An assumption similar to (iii) was also introduced in \cite{SPDE} to obtain stability estimates in the case where the time-dependent coefficients are not differentiable in $t$. 
\end{remark}

\begin{example}
	The most simple example for a family of random paths $(W^{s,t})_{(s,t)\in \Delta_T}$ that satisfies Assumption \ref{a1} with $m=2$ and $\chi(s,t)=(t-s)$ is given by 
	\begin{align*}
		W^{s,t}_r=\frac{r-s}{\sqrt{t-s}}X^{s,t},
	\end{align*}
	where $(X^{s,t})_{(s,t)\in \Delta_T}$ is a family of independent, identically distributed random variables with $\E\left[X^{s,t}\right]=0$, $\text{Var}(X^{s,t})=\sigma>0$ and $|X^{s,t}|\leq C$. Such a family of random variables can be constructed by the Kolmogorov extension theorem.
\end{example}

As already mentioned in the beginning, our setting also covers cubature approximations on Wiener space, which is clear from the next example.

\begin{example}
	If $W:\Omega\rightarrow \mathcal{C}^{1-\text{Höl}}([0,1],\mathbb{R}^d)$ defines a cubature formula of degree $m$ in the sense of the definition in \cite{article}, i.e. $W$ is discrete and it satisfies 
	\begin{align*}
		\E\left[S_{0,1}^m[W]\right]=\E\left[S_{0,1}^m[\circ B]\right],
	\end{align*}
	for some $m\geq 2$, then given a family of independent copies $(\tilde{W}^{s,t})_{(s,t)\in \Delta_T}$ of $W$,
	\begin{align*}
		W^{s,t}_r:=\sqrt{t-s}\tilde{W}^{s,t}(\frac{r-s}{t-s})
	\end{align*}
	satisfies Assumption \ref{a1} with order $m$ and control $\chi(s,t)=(t-s)$.
\end{example}

In order to formulate our main result, we need to introduce the definition of an evolution system and a martingale problem. 

\begin{definition}
	A two parameter family of bounded linear operators $(Q_{s,t})_{s\leq t}$, $s,t\in [0,T]$ on a Banach space $(\mathbb{B},\|\cdot\|)$ is called evolution system, if 
	\begin{enumerate}[label=(\arabic*)]
		\item $Q_{s,s}=Id_{\mathbb{B}}$,
		\item $Q_{s,t}=Q_{s,u}Q_{u,t}$, for $0\leq s\leq u\leq t\leq T$ (evolution property).
	\end{enumerate}
	The evolution system is called strongly continuous, if the map $(s,t)\mapsto Q_{s,t}$ is strongly continuous for $0\leq s\leq t\leq T$. The evolution system is called an evolution system of contractions, if 
	\begin{align*}
		\|Q_{s,t}f\|\leq \|f\|,
	\end{align*}
	for any $f\in \mathbb{B}$ and all $0\leq s\leq t\leq T$. For a strongly continuous evolution system we can define the corresponding family of generators $\mathcal{A}_t:\mathcal{D}(\mathcal{A}_t)\subset \mathbb{B}\rightarrow \mathbb{B}$ by 
	\begin{align*}
		\mathcal{A}_tf:=\lim\limits_{h\downarrow 0}\frac{1}{h}\left(Q_{t,t+h}f-f\right),
	\end{align*}
	on the domain
	\begin{align*}
		\mathcal{D}(\mathcal{A}_t):=\{f\in \mathbb{B}|\lim\limits_{h\downarrow 0}\frac{1}{h}\left(Q_{t,t+h}f-f\right)\text{ exists}\}.
	\end{align*}
	Let 
	\begin{align*}
		\hat{C}(\mathbb{R}^d,\mathbb{R}):=\{f\in \mathcal{C}(\mathbb{R}^d,\mathbb{R})|\forall \epsilon>0,\text{  }\exists K\subset \mathbb{R}^d\text{ compact, s.t. }|f(x)|<\epsilon,\forall x\notin K \}.
	\end{align*}
	A strongly continuous evolution system of contractions on $(\hat{C}(\mathbb{R}^d,\mathbb{R}),|\cdot|_\infty)$ is called a Feller evolution system. An evolution system on $(\hat{C}(\mathbb{R}^d,\mathbb{R}),|\cdot|_\infty)$ is called positivity preserving, if $Q_{s,t}f\geq0$, for $f\geq 0$.
\end{definition}

\begin{definition}
	Let $(\mathcal{A}_t)_{t\in [0,T]}$ be a family of linear operators on $\mathcal{C}_c^\infty (\mathbb{R}^d,\mathbb{R})$ and $x\in \mathbb{R}^d$. Denote by $(X_{t})_{t\in [0,T]}$ the canonical process on $\mathcal{C}([0,T],\mathbb{R}^d)$ and by $(\mathcal{B}_t)_{t\in [0,T]}$ its natural filtration. We call a probability measure $Q_x$ on $(\mathcal{C}([0,T],\mathbb{R}^d),\vee_{t\in [0,T]}\mathcal{B}_t,(\mathcal{B}_t)_{t\in [0,T]})$ a solution to the martingale problem $(x,(\mathcal{A}_t)_{t\in [0,T]},\mathcal{C}_c^\infty(\mathbb{R}^d,\mathbb{R}))$, if $Q_x[X_0=x]=1$ and for any $f\in \mathcal{C}_c^\infty(\mathbb{R}^d,\mathbb{R})$ the process 
	\begin{align*}
		f(X_t)-f(x)-\int_{0}^{t}\mathcal{A}_{s}f(X_s)ds
	\end{align*}
	is a $(Q_x,(\mathcal{B}_t)_{t\in [0,T]})$ martingale.
\end{definition}

\subsection{Main Result}

Let $V_0,...,V_d\in \mathcal{C}^\infty_b(\mathbb{R}^d,\mathbb{R}^d)$ be smooth vector fields on $\mathbb{R}^d$. For any $(s,t)\in \Delta_T$ and $f\in \mathcal{C}(\mathbb{R}^d,\mathbb{R})$ we define
\begin{align*}
	Q_{s,t}f(x):=\E\left[f(X_t^{s,t,x})\right],
\end{align*}
where $X^{s,t,x}$ is the solution to

\begin{equation}\label{state}
	\begin{aligned}
		dX^{s,t;x}_r&= \sum_{i=0}^{d}V_i(X^{s,t;x}_r)dW^{s,t,i}_r\,,
		\quad r\in [s,t]
		\\
		X^{s,t;s}_s&=x.
	\end{aligned}
\end{equation}

Given $(s,t)\in \Delta_T$ and any partition $\pi=\{s=t_0<t_1<...<t_k=t\}$, we define
\begin{align*}
	Q_{s,t}^\pi:=\prod_{i=0}^{k-1}Q_{t_i,t_{i+1}}.
\end{align*}

The main Theorem of this paper is as follows.

\begin{theorem}
	\label{thm:main}
	Suppose that \( (W^{s,t})_{(s,t)\in \Delta_T}\) is a family of random paths which satisfy \cref{a1}, with control $\chi$, coefficients $A:[0,T]\rightarrow \mathbb{R}^{d\times d}$ and \( m\ge2.\) Then there exists a unique strongly continuous, positivity preserving Feller evolution system $(P_{s,t})_{(s,t)\in \Delta_T}$ on $\hat{C}(\mathbb{R}^d,\mathbb{R})$, such that for any $(s,t)\in \Delta_T$ and any sequence of partitions $(\pi_n)_n$ of $[s,t]$ with $|\pi_n|\rightarrow 0$ it holds
	\begin{align*}
		\sup_{\substack{f\in \hat{C}(\mathbb{R}^d,\mathbb{R})\cap \mathcal{C}_b^{m+1}(\mathbb{R}^d,\mathbb{R})\\|f|_{\mathcal{C}_b^{m+1}}\leq 1}} |Q_{s,t}^{\pi_n} f-P_{s,t}f|_\infty\rightarrow 0,
	\end{align*}
	as $n\rightarrow \infty$. The corresponding family of generators  $(\mathcal{A}_t)_{t\in [0,T]}$ for $(P_{s,t})_{(s,t)\in \Delta_T}$ satisfies
	\begin{align*}
		\mathcal{C}_c^\infty(\mathbb{R}^d,\mathbb{R})\subseteq \mathcal{D}(\mathcal{A}_t)
	\end{align*}
	and
	\begin{align*}
		\mathcal{A}_tf=V_0f+\frac{1}{2}\sum_{i=1}^{d}\sum_{j=1}^{d}a_{ij}(t)V_iV_jf,
	\end{align*}
	for any $f\in \mathcal{C}_c^\infty(\mathbb{R}^d,\mathbb{R})$. For any partition $\pi=\{s=t_0<t_1<...<t_k=t\}$ the following estimate holds true:
	\begin{align*}
		\sup_{\substack{f\in \hat{C}(\mathbb{R}^d,\mathbb{R})\cap \mathcal{C}_b^{m+1}(\mathbb{R}^d,\mathbb{R})\\|f|_{\mathcal{C}_b^{m+1}}\leq 1}} |Q_{s,t}^\pi f-P_{s,t}f|_\infty&\leq  2^{\frac{m+1}{2}}e^{\tilde{C}\tilde{\chi}(s,t)}\zeta(\frac{m+1}{2})\chi(s,t)\max_{1\leq i\leq k}\{\tilde{\chi}^{\frac{m-1}{2}}(t_i,t_{i+1})\},
	\end{align*}
	where $\tilde{\chi}(s,t)=\chi(s,t)+(t-s)\sqrt{\chi(s,t)}+(t-s)$. If in addition $a_v(t,x)=(a_v(t,x)_{ij})_{1\leq i,j\leq d}:=(a_{ij}(t)V_i(x)\frac{\partial}{\partial x}V_j(x))_{1\leq i, j\leq d}$ is non-degenerate, i.e.
	\begin{align*}
		\inf_{\|z\|\leq R}\inf_{x\in \mathbb{R}^d,\|x\|=1}x^Ta_v(t,z)x>0,
	\end{align*}
	for all $R>0$ and all $t\in [0,T]$, then for any $x\in \mathbb{R^d}$ there exists a unique solution $(X_t^x)_{t\in [0,T]}$ to the martingale problem $(x,(\mathcal{A}_t)_{t\in [0,T]},\mathcal{C}_c^\infty(\mathbb{R}^d,\mathbb{R}))$ and it holds for any $f\in \mathcal{C}_c^\infty(\mathbb{R}^d,\mathbb{R})$ and $x\in \mathbb{R}^d$
	\begin{align*}
		P_{0,T}f(x)=\mathbb{E}\left[f(X_T^x)\right].
	\end{align*}
\end{theorem}

\section{Proof of the main result}

The main ingredient to the proof of our main Theorem \ref{thm:main} will be the multiplicative sewing lemma Theorem \ref{sewing}. To verify the assumptions needed in Theorem \ref{sewing} we will split up the proof of our main result into the following three parts: In the first part we will derive some stability estimates for the family of operators $(Q_{s,t})_{(s,t)\in \Delta_T}$ in order to control $Q^\pi_{s,t}$ uniformly w.r.t.\ the partition $\pi$. In the second part we will establish  suitable estimates for the commutators $Q_{s,t}-Q_{s,u}Q_{u,t}$. In the last step we finally apply Theorem \ref{sewing} and conclude our proof by determining the corresponding generator of the limiting evolution system on a subset of its domain.

\subsection{Stability Analysis}
\label{ssec:stability}

As already mentioned, the main goal of this section is to provide some stability estimates for $Q_{s,t}$, in order to derive some growth rate $\epsilon$, such that for some appropriate norm $|\cdot |$ it holds $|Q_{s,t}^\pi|\leq \epsilon(s,t)$, independent of the given partition $\pi$ of $[s,t]$. 

Our analysis of $Q_{s,t}$ heavily relies on the regularity of the solution map $\phi:x\mapsto X^{s,t,x}$. In the following we will provide some useful results regarding the latter, which we will need later on.

\begin{lemma}\label{boundSol}
	Let $(s,t)\in \Delta_T$ and $x\in \mathbb{R}$ be fixed. Then the equation (\ref{state}) has a unique (pathwise) solution $X^{s,t,x}\in \mathcal{C}^{1\text{-var}}([s,t],\mathbb{R}^d)$. Furthermore, for any $p\geq 1$ the following pathwise estimate holds true:
	\begin{align*}
		\sup_{u\in [s,t]}|X_{s,u}^{s,t,x}|^p\leq \tilde{C}(p,d,\left|V\right|_{\mathcal{C}^1_b})\tilde{\chi}(s,t)e^{\tilde{C}(p,d,\left|V\right|_{\mathcal{C}^1_b})\tilde{\chi}(s,t)},
	\end{align*}
	for some $\tilde{C}(p,d,\left|V\right|_{\mathcal{C}^1_b})>0$, where $X_{s,u}^{s,t,x}:=X_u^{s,t,x}-x$.
\end{lemma}

\begin{proof}
	Pathwise existence and uniqueness of the solution can be proven in a similar way as it is done in Theorem 3.7 and Theorem 3.8 of \cite{Fritz}. Let $s\leq u\leq t$ be fixed. Then we get
	\begin{align*}
		|X_{s,u}^{s,t,x}|^p
		&=p\sum_{i=0}^{d}\int_{s}^{u}|X_{s,r}^{s,t,x}|^{p-2}X_{s,r}^{s,t,x}\cdot V_i(X_r^{s,t,x})dW_r^{s,t,i}
	\end{align*}
	Furthermore we have
	\begin{align*}
		&|X_{s,r}^{s,t,x}|^{p-2}X_{s,r}^{s,t,x}\cdot V_i(X_r^{s,t,x})\\
		&= \sum_{j=0}^{d}\int_{s}^{r}  
		|X_{s,l}^{s,t,x}|^{p-2}
		X_{s,l}^{s,t,x}\cdot DV_i(
		X_{l}^{s,t,x})V_j(
		X_{l}^{s,t,x})dW^{s,t,j}_l \\
		&\quad + (p-2)\sum_{j=0}^{d}\int_{s}^{r}  |X_{s,l}^{s,t,x}|^{p-4}\left(X_{s,l}^{s,t,x}\cdot V_j(X_l^{s,t,x}) \right)\left(X_{s,l}^{s,t,x}\cdot V_i(X_l^{s,t})\right)dW^{s,t,j}_l\\
		&\quad + \sum_{j=0}^{d}\int_{s}^{r}  |X_{s,l}^{s,t,x}|^{p-2}V_i(X_l^{s,t,x})\cdot V_j(X_l^{s,t,x})dW^{s,t,j}_l.
	\end{align*}
	Together this leads to the expansion
	\begin{align*}
		&|X_{s,u}^{s,t,x}|^p\\
		&=p\int_{s}^{u}|X_{s,r}^{s,t,x}|^{p-2}X_{s,r}^{s,t,x}\cdot V_0(X_r^{s,t,x})dr\\
		&\quad +p\sum_{i=1}^{d}\sum_{j=0}^{d}\int_{s}^{u}\int_{s}^{r}  
		|X_{s,l}^{s,t,x}|^{p-2}
		X_{s,l}^{s,t,x}\cdot DV_i(
		X_{l}^{s,t,x})V_j(
		X_{l}^{s,t,x})dW^{s,t,j}_ldW_r^{s,t,i} \\
		&\quad + p(p-2)\sum_{i=1}^{d}\sum_{j=0}^{d}\int_{s}^{u}\int_{s}^{r}  |X_{s,l}^{s,t,x}|^{p-4}\left(X_{s,l}^{s,t,x}\cdot V_j(X_l^{s,t,x}) \right)\left(X_{s,l}^{s,t,x}\cdot V_i(X_l^{s,t})\right)dW^{s,t,j}_ldW_r^{s,t,i}\\
		&\quad + p\sum_{i=1}^{d}\sum_{j=0}^{d}\int_{s}^{u}\int_{s}^{r}  |X_{s,l}^{s,t,x}|^{p-2}V_i(X_l^{s,t,x})\cdot V_j(X_l^{s,t,x})dW^{s,t,j}_ldW_r^{s,t,i}.
	\end{align*}
	By Assumption \ref{a1} there exists a $C>0$, such that for $i=1,...,d$
	\begin{align*}
		|\frac{\partial}{\partial r}W_r^{s,t,i}|\leq C\frac{\chi(s,t)^{1/2}}{t-s}
	\end{align*}
	and 
	\begin{align*}
		\int_{s}^{t}d|W_r^{s,t,i}|\leq C\chi(s,t)^{1/2}.
	\end{align*}
	Therefore an application of Fubini's theorem and Cauchy-Schwarz inequality yields
	\begin{align*}
		|X_{s,u}^{s,t,x}|^p&\leq Cp\int_{s}^{u}|X_{s,r}^{s,t,x}|^{p-1}|V_0|_\infty  dr\\
		&\quad + C^2p\frac{\chi(s,t)}{t-s}\sum_{i=1}^{d}\sum_{j=1}^{d}\int_{s}^{u}|X_{s,r}^{s,t,x}|^{p-1}|DV_i|_\infty |V_j|_\infty dr\\
		&\quad + C^2p(p-1)\frac{\chi(s,t)}{t-s}\sum_{i=1}^{d}\sum_{j=1}^{d}\int_{s}^{u}|X_{s,r}^{s,t,x}|^{p-2}|V_i|_\infty |V_j|_\infty dr\\
		&\quad + Cp\sqrt{\chi(s,t)}\sum_{i=1}^{d}\int_{s}^{u}|X_{s,r}^{s,t,x}|^{p-1}|DV_i|_\infty |V_0|_\infty dr\\
		&\quad + Cp(p-1)\sqrt{\chi(s,t)}\sum_{i=1}^{d}\int_{s}^{u}|X_{s,r}^{s,t,x}|^{p-2}|V_i|_\infty |V_0|_\infty dr.
	\end{align*}
	Now Young's inequality leads to 
	\begin{align*}
		|X_{s,u}^{s,t,x}|^p&\leq \tilde{C}(p,d,\left|V\right|_{\mathcal{C}^1_b})\tilde{\chi}(s,t)\\
		&\quad +\tilde{C}(p,d,\left|V\right|_{\mathcal{C}^1_b})\int_{s}^{u}|X_{s,r}^{s,t,x}|^{p} dr\\
		&\quad + \tilde{C}(p,d,\left|V\right|_{\mathcal{C}^1_b})\frac{\chi(s,t)}{t-s}\sum_{j=1}^{d}\int_{s}^{u}|X_{s,r}^{s,t,x}|^{p} dr\\
		&\quad + \tilde{C}(p,d,\left|V\right|_{\mathcal{C}^1_b})\sqrt{\chi(s,t)}\int_{s}^{u}|X_{s,r}^{s,t,x}|^{p}dr.
	\end{align*}
	Taking the supremum over $[s,t]$, a simple application of Gronwall inequality shows 
	\begin{align*}
		\sup_{u\in [s,t]}|X_{s,u}^{s,t,x}|^p\leq \tilde{C}(p,d,\left|V\right|_{\mathcal{C}^1_b})\tilde{\chi}(s,t)e^{\tilde{C}(p,d,\left|V\right|_{\mathcal{C}^1_b})\tilde{\chi}(s,t)}.
	\end{align*}
\end{proof}

By formal differentiation, one can derive control ODEs satisfied pathwise by the directional derivatives of the solution map. The following Lemma provides estimates for this type of random ODE.

\begin{lemma}\label{VarEst}
	Let $(s,t)\in \Delta_T$ be fixed. For $p\geq 1$, $x\in \mathbb{R}^d$ and $i=0,...,d$, let $B^{s,t,x,i}\in L^p(\Omega,\mathcal{C}_b([s,t],\mathbb{R}^d)\cap \mathcal{C}^{\text{1-var}}([s,t],\mathbb{R}^d))$ and $\eta\in L^p(\Omega,\mathbb{R}^d)$, such that $B^{s,t,x,i}_s$ is independent of $\eta$ and $W^{s,t,i}$. Moreover we assume, that
	\begin{align*}
		\sup_{r\in [s,t]}\left|B_r^{s,t,x,i}-B_s^{s,t,x,i}\right|\leq \hat{C}^{s,t,x} \chi(s,t)^{1/2},
	\end{align*}
	$\mathbb{P}$-almost surely, for some $\hat{C}^{s,t,x}\in L^p(\Omega,\mathbb{R})$. 
	Then the equation 
	\begin{equation}\label{variational}
		\begin{aligned}
			dZ^{s,t,B,x}_r&=\sum_{i=0}^d B_r^{s,t,x,i}dW_r^{s,t,i}
			+ \sum_{i=0}^d DV_i(X_r^{s,t,x})Z_r^{s,t,B,x}dW_r^{s,t,i}\\
			Z^{s,t,B,x}_s&=\eta
		\end{aligned}
	\end{equation}
	has a unique solution $Z^{s,t,B,x}\in L^p(\Omega,\mathcal{C}_b([s,t],\mathbb{R}^d)\cap C^{1-var}([s,t],\mathbb{R}^d))$, such that the following estimate holds true $\mathbb{P}$-almost surely:
	\begin{equation}\label{variational_bound1}
		\begin{aligned}
		\sup_{s\leq u\leq t}|Z_u^{s,t,B,x}|^p&\leq \tilde{C}(p,d,\left|V\right|_{\mathcal{C}^2_b})\left(|\eta|^p+|B^{s,t,x}_s|^p+|\hat{C}^{s,t,x}|^p\right)e^{\tilde{C}(p,d,\left|V\right|_{\mathcal{C}^2_b})\tilde{\chi}(s,t)}.
		\end{aligned}
	\end{equation} 
	Furthermore it holds
	\begin{equation}
		\begin{aligned}
			&\E\left[|Z_t^{s,t,B,x}|^p\right]\leq\|\eta\|_{L^p}^pe^{\tilde{C}(p,d,\left|V\right|_{\mathcal{C}^2_b})\tilde{\chi}(s,t)}\\
			&\quad +\tilde{C}(p,d,\left|V\right|_{\mathcal{C}^2_b})\left(\|B^{s,t,x}_s\|_{L^p}^p+\|\hat{C}^{s,t,x}\|_{L^p}^p\right)\tilde{\chi}(s,t)e^{\tilde{C}(p,d,\left|V\right|_{\mathcal{C}^2_b})\tilde{\chi}(s,t)}.
		\end{aligned}
	\end{equation}
\end{lemma}

\begin{proof}
	In the following we will omit the dependencies on $s,t$ for $B$ and $Z$. Pathwise existence and uniqueness of the solution can be again proven in a similar way as it is done in Theorem 3.7 and Theorem 3.8 of \cite{Fritz}. Now for any $s\leq u\leq t$ we have
	\begin{equation}
		\begin{aligned}
			|Z_u^x|^p&=\left|\eta\right|^p+p\sum_{i=0}^{d}\int_{s}^{u} |Z_r^x|^{p-2}Z_r^x\cdot B_r^{i,x}dW_r^{s,t,i}\\
			&\quad +p\sum_{i=0}^{d} \int_{s}^{u}|Z_l^x|^{p-2}Z_r^x\cdot DV_i(X_r^x)Z_r^x dW_r^{s,t,i}.
		\end{aligned}
	\end{equation}
	By Assumption \ref{a1}, there exists a $C>0$, such that for $i=1,...,d$
	\begin{align*}
		|\frac{\partial}{\partial r}W_r^{s,t,i}|\leq C\frac{\chi(s,t)^{1/2}}{t-s}.
	\end{align*}
	Therefore it holds
	\begin{align*}
		\sum_{i=0}^{d}\int_{s}^{u} |Z_r^x|^{p-2}Z_r^x\cdot B_r^{i,x}dW_r^{s,t,i}&=\sum_{i=0}^{d}\int_{s}^{u} |Z_r^x|^{p-2}Z_r^x\cdot (B_r^{i,x}-B_s^{i,x})dW_r^{s,t,i}\\
		&\quad +\sum_{i=0}^{d}\int_{s}^{u} |Z_r^x|^{p-2}Z_r^x\cdot B_s^{i,x}dW_r^{s,t,i}\\
		&\leq C\frac{\chi(s,t)^{1/2}}{t-s}\sum_{i=0}^{d}\int_{s}^{u} |Z_r^x|^{p-1} |B_r^x-B_s^x|dr\\
		&\quad  +\sum_{i=0}^{d}\int_{s}^{u} |Z_r^x|^{p-2}Z_r^x\cdot B_s^{i,x}dW_r^{s,t,i}.
	\end{align*}
	Now we have 
	\begin{equation}
		\begin{aligned}
		&|Z_r^x|^{p-2}Z_r^x\cdot B_s^{i,x}\\
		&=\left|\eta\right|^{p-2}\eta \cdot B_s^{i,x} + (p-2)\sum_{j=0}^{d}\int_{s}^{r}  |Z_l^x|^{p-4}\left(Z_l^x\cdot B_l^{j,x}\right) \left(Z_l^x\cdot B_s^{i,x}\right)dW^{s,t,j}_l\\
		&\quad + \sum_{j=0}^{d}\int_{s}^{r}  |Z_l^x|^{p-2}B_l^{j,x}\cdot B_s^{i,x}dW^{s,t,j}_l\\
		&\quad + (p-2)\sum_{j=0}^{d}\int_{s}^{r}  |Z_l^x|^{p-4}\left(Z_l^x\cdot DV_j(X_l^x)Z_l^x \right)\left(Z_l^x\cdot B_s^{i,x}\right)dW^{s,t,j}_l\\
		&\quad + \sum_{j=0}^{d}\int_{s}^{r}  |Z_l^x|^{p-2}DV_j(X_l^x)Z_l^x\cdot B_s^{i,x}dW^{s,t,j}_l
		\end{aligned}
	\end{equation}
	and furthermore it holds 
	\begin{equation}
		\begin{aligned}
			&|Z_r^x|^{p-2}Z_r^x\cdot DV_i(X_r^{x})Z_r^x\\
			&=\left|\eta\right|^{p-2}\eta \cdot DV_i(x)\eta + \sum_{j=0}^{d}\int_{s}^{r}  |Z_l^x|^{p-2}Z_l^x\cdot DV_i(X_l^{x})B_l^{j,x}dW^{s,t,j}_l\\
			&\quad + \sum_{j=0}^{d}\int_{s}^{r}  |Z_l^x|^{p-2}Z_l^x\cdot DV_i(X_l^{x})DV_j(X_l^{x})Z_ldW^{s,t,j}_l \\
			&\quad + \sum_{j=0}^{d}\int_{s}^{r}  |Z_l^x|^{p-2}Z_l^x\cdot D^2V_i(X_l^{x})V_j(X_l^{x})Z_l^xdW^{s,t,j}_l \\
			&\quad + (p-2)\sum_{j=0}^{d}\int_{s}^{r}  |Z_l^x|^{p-4}\left(Z_l^x\cdot B_l^{j,x}\right) \left(Z_l^x\cdot DV_i(X_l^{x})Z_l^x\right)dW^{s,t,j}_l\\
			&\quad + \sum_{j=0}^{d}\int_{s}^{r}  |Z_l^x|^{p-2}B_l^{j,x}\cdot  DV_i(X_l^{x})Z_l^{x}dW^{s,t,j}_l\\
			&\quad + (p-2)\sum_{j=0}^{d}\int_{s}^{r}  |Z_l^x|^{p-4}\left(Z_l^x\cdot DV_j(X_l^{x})Z_l^x \right)\left(Z_l^x\cdot DV_i(X_l^{x})Z_l^x\right)dW^{s,t,j}_l\\
			&\quad + \sum_{j=0}^{d}\int_{s}^{r}  |Z_l^x|^{p-2}DV_j(X_l^{x})Z_l^x \cdot  DV_i(X_l^{x})Z_ldW^{s,t,j}_l.
		\end{aligned}
	\end{equation}
	Using Assumption \ref{a1} again,  there exists a $C>0$, such that for $i=1,...,d$
	\begin{align*}
		\int_{s}^{t}d|W_r^{s,t,i}|\leq C\chi(s,t)^{1/2}.
	\end{align*}
	Using the assumption on $B$, Cauchy-Schwarz inequality and Fubini's theorem, we arrive at 
	\begin{equation}
		\begin{aligned}
			|Z_u^x|^p
			&\leq\left|\eta\right|^p+p\left|\eta\right|^{p-2}\sum_{i=1}^{d}\eta \cdot DV_i(x)\eta \int_{s}^{u}dW^{s,t,i}_l +p\left|\eta\right|^{p-2}\sum_{i=1}^{d}\eta \cdot B_s^i\int_{s}^{u}dW^{s,t,i}_l\\
			&\quad + p\int_{s}^{u} |Z_l^x|^{p-1}\left|B_l^{0,x}\right|dl +p \int_{s}^{u}|Z_l^x|^{p}\left|DV_0\right| dl+pC\frac{\chi(s,t)}{t-s}\sum_{i=0}^{d}\int_{s}^{u}  |Z_l^x|^{p-1}\left|\hat{C}^x\right|dl\\
			&\quad + p^2C^2\frac{\chi(s,t)}{t-s}\sum_{i=1}^{d}\sum_{j=1}^{d}\int_{s}^{u} |Z_l^x|^{p-1}\left|B_l^{j,x}\right| \left| DV_i\right|_\infty dl\\
			&\quad +p^2C^2\sqrt{\chi(s,t)}\sum_{i=1}^{d}\int_{s}^{u} |Z_l^x|^{p-1}\left|B_l^{0,x}\right| \left| DV_i\right|_\infty dl\\
			&\quad + p^2C^2\frac{\chi(s,t)}{t-s}\sum_{i=1}^{d}\sum_{j=1}^{d}\int_{s}^{u} |Z_l^x|^{p}\left|DV_j \right|_\infty\left| DV_i\right|_\infty dl\\
			&\quad +p^2C^2\sqrt{\chi(s,t)}\sum_{i=1}^{d}\int_{s}^{u} |Z_l^x|^{p}\left|DV_j \right|_\infty\left| DV_i\right|_\infty dl\\
			&\quad + (p^2-p)C^2\frac{\chi(s,t)}{t-s}\sum_{i=1}^{d}\sum_{j=1}^{d}\int_{s}^{u}  |Z_l^x|^{p-2}\left|B_l^{i,x}\right|\left|B_l^{j,x}\right|dl\\
			&\quad +(p^2-p)C^2\sqrt{\chi(s,t)}\sum_{i=1}^{d}\int_{s}^{u}  |Z_l^x|^{p-2}\left|B_l^{i,x}\right|\left|B_l^{0,x}\right|dl\\
			&\quad + (p^2-p)C^2\frac{\chi(s,t)}{t-s}\sum_{i=1}^{d}\sum_{j=1}^{d}\int_{s}^{t}  |Z_l^x|^{p-1} |DV_j|_\infty \left|B_l^{i,x}\right|dl\\
			&\quad +(p^2-p)C^2\sqrt{\chi(s,t)}\sum_{i=1}^{d}\int_{s}^{t}  |Z_l^x|^{p-1} |DV_j|_\infty \left|B_l^{i,x}\right|dl\\
			&\quad + pC^2\frac{\chi(s,t)}{t-s}\sum_{i=1}^{d}\sum_{j=1}^{d}\int_{s}^{t} |Z_l^x|^{p} \left|D^2V_i\right|_\infty \left|V_j\right|_\infty dl\\
			&\quad +pC^2\sqrt{\chi(s,t)}\sum_{i=1}^{d}\int_{s}^{t} |Z_l^x|^{p} \left|D^2V_i\right|_\infty \left|V_j\right|_\infty dl.
		\end{aligned}
	\end{equation}
	An application of Young inequality leads to 
	\begin{equation}\label{young1}
		\begin{aligned}
			|Z_u^x|^p
			&\leq\left|\eta\right|^p+p\left|\eta\right|^{p-2}\sum_{i=1}^{d}\eta \cdot DV_i(x)\eta \int_{s}^{u}dW^{s,t,i}_l +p\left|\eta\right|^{p-2}\sum_{i=1}^{d}\eta \cdot B_s^{i,x}\int_{s}^{u}dW^{s,t,i}_l\\
			&\quad + \tilde{C}(p,d,\left|V\right|_{\mathcal{C}^2_b})\frac{\chi(s,t)}{t-s}\left(\int_{s}^{u} |Z_r^x|^{p}dr + \int_{s}^{u}\left|B_r^x\right|^p dr +\left|\hat{C}^x\right|^p (t-s)\right)\\
			&\quad +\tilde{C}(p,d,\left|V\right|_{\mathcal{C}^2_b})\sqrt{\chi(s,t)}\left(\int_{s}^{u} |Z_r^x|^{p}dr+\int_{s}^{u}|B_r^x|^pdr\right)\\
			&\quad +\tilde{C}(p,d,\left|V\right|_{\mathcal{C}^2_b})\left(\int_{s}^{u} |Z_r^x|^{p}dr+\int_{s}^{u}|B_r^{0,x}|^pdr\right),
		\end{aligned}
	\end{equation}
	for some constant $\tilde{C}(p,d,\left|V\right|_{\mathcal{C}^2_b})>0$. Taking the expectation with the supremum over $u\in [s,t]$ and using Assumption \ref{a1} one more time, an application of Gronwall lemma shows that 
	\begin{align*}
		\sup_{s\leq u\leq t}|Z_u^x|^p
		&\leq \tilde{C}(p,d,\left|V\right|_{\mathcal{C}^2_b})\left(|\eta|^p+|B^x_s|^p+|\hat{C}^x|^p\right)e^{\tilde{C}(p,d,\left|V\right|_{\mathcal{C}^2_b})\tilde{\chi}(s,t)}.
	\end{align*}
	Fixing $u=t$ in (\ref{young1}) we can use that $\E\left[\int_{s}^{t}dW^{s,t,i}\right]=0$, then taking the expectation, a simple application of Gronwall lemma leads to 
	\begin{equation}
		\begin{aligned}
			\E\left[|Z_t^x|^p\right]&\leq\|\eta\|_{L^p}^pe^{\tilde{C}(p,d,\left|V\right|_{\mathcal{C}^2_b})\tilde{\chi}(s,t)}\\
			&\quad +\tilde{C}(p,d,\left|V\right|_{C^2_b})\left(\|B^{x}_s\|_{L^p}^p+\|\hat{C}^{x}\|_{L^p}^p\right)\tilde{\chi}(s,t)e^{\tilde{C}(p,d,\left|V\right|_{\mathcal{C}^2_b})\tilde{\chi}(s,t)}.
		\end{aligned}
	\end{equation}
\end{proof}

\begin{lemma}\label{directional}
	For any $(s,t)\in \Delta_T$, the map
	\begin{align*}
		\phi^{s,t}:\mathbb{R}^d\rightarrow \mathcal{C}^{1-\text{var}}([s,t],\mathbb{R}^d), \quad 
		x\mapsto X^{s,t,x}
	\end{align*}
	has directional derivatives up to an arbitrary order, i.e. for all $k\in \mathbb{N}$ and $(\eta_i)_{1\leq i\leq k}\subseteq \mathbb{R}^d$
	\begin{align*}
		D^k_{\eta_1,...,\eta_k}\phi^{s,t}(x):=\left(\frac{\partial^k}{\partial \epsilon_1...\partial \epsilon_k}X^{s,t,x+\sum_{j=1}^{k}\epsilon_j \eta_j}\right)|_{\epsilon=0}
	\end{align*}
	exists as a strong limit in the Banach space $\mathcal{C}^{1-\text{var}}([s,t],\mathbb{R}^d)$. If we denote 
	\begin{align*}
		Z_r^{s,t,x,\eta_1,...,\eta_k}:=D^k_{\eta_1,...,\eta_k}\phi^{s,t}(x)_r,
	\end{align*}
	then it holds for $k=1$ and any $p\geq 1$
	\begin{equation}\label{directional1}
		\begin{aligned}
			\sup_{|\eta_1|\leq 1}\sup_{x\in \mathbb{R}^d}\sup_{s\leq r\leq 	t}|Z_r^{s,t,x,\eta_1}|^p&\leq C(p,d,|V|_{\mathcal{C}_b^2})e^{C(p,d,|V|_{\mathcal{C}_b^2})\tilde{\chi}(s,t)},\\
			\sup_{|\eta_1|\leq 1}\E\left[|Z_t^{s,t,x,\eta_1}|^p\right]&\leq 	e^{C(p,d,|V|_{\mathcal{C}_b^2})\tilde{\chi}(s,t)},
		\end{aligned}
	\end{equation}
	for some $C(p,d,|V|_{\mathcal{C}_b^2})>0$. For $k>1$ and any $p\geq 1$ it holds
	\begin{equation}\label{directional2}
		\begin{aligned}
		\sup_{|\eta_1|\leq 1,...,|\eta_k|\leq 1}\sup_{x\in \mathbb{R}^d}\sup_{s\leq r\leq t}|Z_r^{s,t,x,\eta_1,...,\eta_k}|^p&\leq C(p,k,d,|V|_{\mathcal{C}_b^k})e^{C(p,k,d,|V|_{\mathcal{C}_b^k})\tilde{\chi}(s,t)},\\
		\sup_{|\eta_1|\leq 1,...,|\eta_k|\leq 1}\E\left[|Z_t^{s,t,x,\eta_1,...,\eta_k}|^p\right]&\leq C(p,k,d,|V|_{\mathcal{C}_b^k})\tilde{\chi}(s,t)e^{C(p,k,d,|V|_{\mathcal{C}_b^k})\tilde{\chi}(s,t)},
		\end{aligned}
	\end{equation}
	for some $C(p,k,d,|V|_{\mathcal{C}_b^k})>0$. 
\end{lemma}

\begin{proof}
	Since $V$ is a collection of smooth vector fields, the existence and continuity of the directional derivatives of the solution map $\phi :x\mapsto X^{s,t,x}$ follows from (\cite{Fritz}, Proposition 4.6). By formal differentiation one can derive the control ODEs satisfied by the directional derivatives.
	For $k=1$, the directional derivative $Z^{s,t,\eta_1}$ satisfies the equation 
	\begin{equation}
		\begin{aligned}
			dZ^{s,t,x,\eta_j}_r&=\sum_{i=0}^d DV_i(X_r^{s,t,x})Z_r^{s,t,x,\eta_j}dW_r^{s,t,i}\\
			Z^{s,t,x,\eta_j}_s&=\eta_1.
		\end{aligned}
	\end{equation}
	This is an equation of type (\ref{variational}) with $B^{s,t,x}=0$ and therefore (\ref{directional1}) follows from Lemma \ref{VarEst}. \\\\For arbitrary $k>1$, the bound in (\ref{directional2}) will be proven by induction.\\\\
	For $k=2$, the directional derivative $Z^{s,t,\eta_1,\eta_2}$ satisfies the equation 
	\begin{equation}
		\begin{aligned}
			dZ^{s,t,x,\eta_1,\eta_2}_r&=\sum_{i=0}^d D^2V_i(X_r^{s,t,x})Z_r^{s,t,x,\eta_1}Z_r^{s,t,x,\eta_2}dW_r^{s,t,i}\\
			&\quad + \sum_{i=0}^d DV_i(X_r^{s,t,x})Z_r^{s,t,x,\eta_1,\eta_2}dW_r^{s,t,i}\\
			Z^{s,t,x,\eta_1,\eta_2}_s&=0.
		\end{aligned}
	\end{equation}
	If we define
	\begin{align*}
		B^{s,t,x,i}_r:=D^2V_i(X_r^{s,t,x})Z_r^{s,t,x,\eta_1}Z_r^{s,t,x,\eta_2},
	\end{align*}
	then it holds
	\begin{equation}\label{intParts}
		\begin{aligned}
		&D^2V_i(X_r^{s,t,x})Z_r^{s,t,x,\eta_1}Z_r^{s,t,x,\eta_2}=D^2V_i(x)\eta_1\eta_2\\
		&\quad + \sum_{j=0}^{d}\int_{s}^{r}D^2V_i(X_l^{s,t,x})Z_l^{s,t,x,\eta_1}DV_j(X_l^{s,t,x})Z_l^{s,t,x,\eta_2}\dfrac{\partial}{\partial l}W_l^{s,t,j}dl\\
		&\quad + \sum_{j=0}^{d}\int_{s}^{r} D^2V_i(X_l^{s,t,x})DV_j(X_l^{s,t,x})Z_l^{s,t,x,\eta_1}Z_l^{s,t,x,\eta_2}\dfrac{\partial}{\partial l}W_l^{s,t,j}dl\\
		&\quad + \sum_{j=0}^{d}\int_{s}^{r} D^3V_i(X_l^{s,t,x})Z_l^{s,t,x,\eta_1}V_j(X_l^{s,t,x})Z_l^{s,t,x,\eta_2}\dfrac{\partial}{\partial l}W_l^{s,t,j}dl.
		\end{aligned}
	\end{equation}
	Therefore, using Assumption \ref{a1}, we get 
	\begin{align*}
		|B^{s,t,x,i}_r-B^{s,t,x,i}_s|&\leq \hat{C}^{s,t,x}\chi(s,t)^{1/2},
	\end{align*}
	where 
	\begin{align*}
		\hat{C}^{s,t,x}:=C|V|_{\mathcal{C}_b^3}\sup_{s\leq l\leq t}|Z^{s,t,x,\eta_1}_l|\sup_{s\leq l\leq t}|Z^{s,t,x,\eta_2}_l|.
	\end{align*}
	Now (\ref{directional1}) and Lemma \ref{VarEst} concludes (\ref{directional2}).\\\\
	Now assume that (\ref{directional2}) is true for arbitrary, but fixed $k>1$. \\\\The directional derivative $Z^{s,t,x,\eta_1,...,\eta_{k+1}}$ of order $k+1$ satisfies an equation of type (\ref{variational}), where $B^{s,t,x}$ only involves terms of the form
	\begin{equation}\label{coeffTerms}
		F(X^{s,t,x})Z^{s,t,x,\eta_{\alpha_1}}...Z^{s,t,x,\eta_{\alpha_l}},
	\end{equation}
	for some bounded function $F:\mathbb{R}^d\rightarrow L((\mathbb{R}^d)^l,\mathbb{R}^d)$ and multi-indices $\alpha_i\in \mathcal{A}^1_k$, where for $\alpha=(i_1,...,i_l)$ we define $Z^{s,t,x,\eta_{\alpha}}=Z^{s,t,x,\eta_1...\eta_l}$. Terms of the form (\ref{coeffTerms}) only involve directional derivatives of order strictly less than $k+1$. Therefore we can conclude that 
	\begin{align*}
		\sup_{x\in \mathbb{R}^d}\sup_{s\leq r\leq 	t}|F(X^{s,t,x}_r)Z^{s,t,x,\eta_{\alpha_1}}_r...Z^{s,t,x,\eta_{\alpha_l}}_r|^p \leq C(p,k,d,|V|_{\mathcal{C}_b^k})e^{C(p,k,d,|V|_{\mathcal{C}_b^k})\tilde{\chi}(s,t)}
	\end{align*}
	and
	\begin{align*}
		\E\left[|F(X^{s,t,x}_t)Z^{s,t,x,\eta_{\alpha_1}}_t...Z^{s,t,x,\eta_{\alpha_l}}_t|^p\right] \leq C(p,k,d,|V|_{\mathcal{C}_b^k})\tilde{\chi}(s,t)e^{C(p,k,d,|V|_{\mathcal{C}_b^k})\tilde{\chi}(s,t)}.
	\end{align*}
	Terms which involve at least one directional derivative of order two or higher satisfy
	\begin{equation}
		F(X_s^{s,t,x})Z^{s,t,x,\eta_{\alpha_1}}_s...Z^{s,t,x,\eta_{\alpha_l}}_s=0.
	\end{equation}
	For terms of the form (\ref{coeffTerms}) that only involve directional derivatives of order one, we can argue in a similar way as in the case $k=2$ to show that 
	\begin{align*}
		\E\left[|F(X^{s,t,x}_t)Z^{s,t,x,\eta_1}_t...Z^{s,t,x,\eta_{k+1}}_t-F(x)\eta_1...\eta_{k+1}|^p\right]\leq C(p,k,d,|V|_{\mathcal{C}_b^k})\tilde{\chi}(s,t)e^{C(p,k,d,|V|_{\mathcal{C}_b^k})\tilde{\chi}(s,t)}. 
	\end{align*}
	All together, Lemma \ref{VarEst} concludes (\ref{directional2}) for directional derivatives of order $k+1$. This finishes the proof.
\end{proof}

We will now show that $Q_{s,t}$ is for each \( 0\le s\le t\le T\) a well-defined element of the operator algebra $L(\hat{C}(\mathbb{R}^d,\mathbb{R}))\cap L(\mathcal{C}^{m+1}_b(\mathbb{R}^d,\mathbb{R}))$. We also provide some stability estimates, in order to find some moderate growth rate for $|Q_{s,t}^\pi|$, which is independent of the partition $\pi$.

\begin{lemma}\label{stability}
	We fix $(s,t)\in \Delta_T$ and $K\in \mathbb{N}_0$. For any  $f\in \mathcal{C}^K_b(\mathbb{R}^d,\mathbb{R})$,  the function $Q_{s,t}f$ is $K$- times continuous differentiable with
	\begin{equation}
		\begin{aligned}
			\sup_{l=0,...,K}|D^l\left(Q_{s,t}f\right)|_\infty&\leq \sup_{l=0,...,K}| D^lf|_\infty e^{\tilde{C}(K,d,\left|V\right|_{C^K_b})\tilde{\chi}(s,t)},
		\end{aligned}
	\end{equation}
	for some $\tilde{C}(K,d,\left|V\right|_{\mathcal{C}^K_b})>0$. 
\end{lemma}

\begin{remark}
	This shows in particular, that for any $(s,t)\in \Delta_T$ and $K\in \mathbb{N}$, the operator $Q_{s,t}$ is an element of $ L(\mathcal{C}^K_b(\mathbb{R}^d,\mathbb{R}))$.
\end{remark}

\begin{proof}
	For any $f\in \mathcal{C}_b^0(\mathbb{R}^d,\mathbb{R})$ it holds
	\begin{align*}
		\sup_{x\in \mathbb{R}^d}|Q_{s,t}f(x)|\leq \sup_{x\in \mathbb{R}^d}|f(x)|
	\end{align*}
	and due to the continuity of the solution map $\phi^{s,t}:x\mapsto X^{s,t,x}$ we obtain $Q_{s,t}f\in \mathcal{C}_b^0(\mathbb{R}^d,\mathbb{R})$.\\\\
	Let us now fix some $0<k\leq K$ and $f\in \mathcal{C}_b^K(\mathbb{R}^d,\mathbb{R})$. With the same notation as in Lemma \ref{directional}, we denote by $Z^{s,t,x,\eta_1,...,\eta_k}$ the directional derivative $\partial_{\eta_1}...\partial_{\eta_k} \phi^{s,t}$ of $\phi^{s,t}$ in direction $\eta_1,...,\eta_k$ and for $\alpha=(i_1,...,i_k)\in \mathcal{A}^1$ we define $Z^{s,t,x,\eta_\alpha}:=Z^{s,t,x,\eta_{i_1},...,\eta_{i_k}}$, for given directions $\eta=(\eta_i)_{i\in \mathbb{N}}$. \\\\
	Since $f$ is $K$-times continuous differentiable with bounded derivatives up to order $K$, we can conclude with Lemma \ref{directional} that differentiation under the expectation is justified and therefore the desired differentiability of $Q_{s,t}f$ follows.\\\\ 
	\\
	Now Cauchy-Schwarz inequality leads to
	\begin{equation}\label{stability_derivative}
		\begin{aligned}
			&|D^k\left(Q_{s,t}f\right)|_\infty\\
			&=\sup_{x\in \mathbb{R}^d}\sup_{\eta:|\eta_1|\leq 1,...,|\eta_k|\leq1}\left|D^k\left(Q_{s,t}f\right)(x)\eta_1...\eta_k\right|\\
			&\leq \sup_{x\in \mathbb{R}^d}\sum_{l=1}^{k}S(k,l)\sup_{\eta:|\eta_1|\leq 1,...,|\eta_k|\leq1}\sup_{ \underset{|\alpha^1|+...+|\alpha^l|=k}{\alpha^1,...,\alpha^l\in\mathcal{A}^1(k-l+1)}}^{}\left|\E\left[D^lf(X_t^{s,t,x})Z_t^{s,t,x,\eta_{\alpha^1}}...Z_t^{s,t,x,\eta_{\alpha^l}}\right]\right|\\
			&\leq \sup_{x\in \mathbb{R}^d}\sum_{l=1}^{k}S(k,l)\sup_{\eta:|\eta_1|\leq 1,...,|\eta_k|\leq1}\sup_{ \underset{|\alpha^1|+...+|\alpha^l|=k}{\alpha^1,...,\alpha^l\in\mathcal{A}^1(k-l+1)}}^{}\E\left[\left|D^lf(X_t^{s,t,x})Z_t^{s,t,x,\eta_{\alpha^1}}...Z_t^{s,t,x,\eta_{\alpha^l}}\right|\right]\\
			&\leq \sup_{l=1,...,K}\left|D^lf\right|_\infty\sup_{x\in \mathbb{R}^d}\sum_{l=1}^{k-1}S(k,l)\sup_{\eta:|\eta_1|\leq 1,...,|\eta_k|\leq1}\sup_{ \underset{|\alpha^1|+...+|\alpha^l|=k}{\alpha^1,...,\alpha^l\in\mathcal{A}^1(k-l+1)}}^{}\E\left[\left|Z_t^{s,t,x,\eta_{\alpha^1}}\right|...\left|Z_t^{s,t,x,\eta_{\alpha^l}}\right|\right]\\
			&\quad + \sup_{l=1,...,K}\left|D^lf\right|_\infty\sup_{x\in \mathbb{R}^d}\sup_{\eta:|\eta_1|\leq 1,...,|\eta_k|\leq1}\sup_{ i_1,...,i_k\in \{1,...,d\}}^{}\E\left[\left|Z_t^{s,t,x,{i_1}}\right|...\left|Z_t^{s,t,x,{i_k}}\right|\right].
		\end{aligned}
	\end{equation}
	where $S(k,l)$ is a Stirling number of the second kind. For $l<k$ let $\alpha^1,...,\alpha^l\in \mathcal{A}^1(k-l+1)$ with $|\alpha^1|+...+|\alpha^l|=k$ and w.l.o.g. $|\alpha^1|=...=|\alpha^r|=1$, $|\alpha^{j}|>1$, for $j=r+1,...,l$ and some $r<l$. By Lemma \ref{directional} we then have
	\begin{align*}
		&\sup_{\eta:|\eta_1|\leq 1,...,|\eta_k|\leq1}\E\left[\left|Z_t^{s,t,x,\eta_{\alpha^1}}\right|...\left|Z_t^{s,t,x,\eta_{\alpha^l}}\right|\right]\\
		&\leq C(d,|V|_{\mathcal{C}_b^2})\sup_{\eta:|\eta_1|\leq 1,...,|\eta_k|\leq1}\E\left[\left|Z_t^{s,t,x,\eta_{\alpha^{r+1}}}\right|...\left|Z_t^{s,t,x,\eta_{\alpha^l}}\right|\right] e^{C(d,|V|_{\mathcal{C}_b^2})\tilde{\chi}(s,t)},
	\end{align*}
	for some $C(d,|V|_{\mathcal{C}_b^2})>0$. Using Cauchy-Schwarz inequality, Young inequality and Lemma \ref{directional} again, we obtain 
	\begin{align*}
		\sup_{\eta:|\eta_1|\leq 1,...,|\eta_k|\leq1}\E\left[\left|Z_t^{s,t,x,\eta_{\alpha^1}}\right|...\left|Z_t^{s,t,x,\eta_{\alpha^l}}\right|\right]
		&\leq C(k,d,|V|_{\mathcal{C}_b^k})\tilde{\chi}(s,t)e^{C(k,d,|V|_{\mathcal{C}_b^k})\tilde{\chi}(s,t)},
	\end{align*} 
	for some $C(k,d,|V|_{\mathcal{C}_b^k})>0$. Furthermore we have 
	\begin{align*}
		\sup_{x\in \mathbb{R}^d}\sup_{\eta:|\eta_1|\leq 1,...,|\eta_k|\leq1}\sup_{ i_1,...,i_k\in \{1,...,d\}}^{}\E\left[\left|Z_t^{s,t,x,{i_1}}\right|...\left|Z_t^{s,t,x,{i_k}}\right|\right]\leq e^{C(d,|V|_{\mathcal{C}_b^2})\tilde{\chi}(s,t)}, 
	\end{align*}
	for some $C(d,|V|_{\mathcal{C}_b^2})>0$. Together we can conclude
	\begin{equation}
		\begin{aligned}
			\sup_{l=0,...,K}|D^lQ_{s,t}f|&\leq \sup_{l=0,...,K}|D^lf|e^{C(K,d,|V|_{\mathcal{C}_b^K})\tilde{\chi}(s,t)},
		\end{aligned}
	\end{equation}
for some $C(K,d,|V|_{\mathcal{C}_b^K})>0$.
\end{proof}

We are now able to derive an exponential growth rate for $Q^\pi_{s,t}$ which is independent of the partition $\pi$.

\begin{lemma}\label{growth}
	Let $(s,t)\in \Delta_2$ and $\pi=\{s=t_0<...<t_k=t\}$ be an arbitrary partition of $[s,t]$. Furthermore let $K\in \mathbb{N}$ and
	\[
	\begin{aligned}
		Q_{s,t}^\pi f=\prod_{i=0}^{k-1}Q_{t_i,t_{i+1}}f,
	\end{aligned}
	\]
	then
	\begin{align*}
		\max_{l=0,...,K}\left(\sup_{|f|_{\mathcal{C}_b^l(\mathbb{R}^d,\mathbb{R})}\leq 1}|Q_{s,t}^\pi f|_{\mathcal{C}^l_b(\mathbb{R}^d,\mathbb{R})}\right)\leq e^{\tilde{C}(K,d,\left|V\right|_{\mathcal{C}^K_b})\tilde{\chi}(s,t)},
	\end{align*}
	for some $\tilde{C}(K,d,\left|V\right|_{\mathcal{C}^K_b})>0$ independent of $s,t$.
\end{lemma}

\begin{proof}
	Let $0\leq m\leq K$. By Lemma \ref{stability} we get inductively
	\begin{equation}
		\begin{aligned}
			|Q_{s,t}^\pi f|_{\mathcal{C}^m_b(\mathbb{R}^d,\mathbb{R})}&= \sup_{r=0,...,m}|D^rQ_{s,t}^\pi f|_\infty\\
			&= \sup_{r=0,...,m}|D^rQ_{s,t_{1}}Q_{t_1,t}^{\pi\setminus \{s,t_1\}} f|_\infty\\
			&\leq \sup_{r=0,...,m}|D^rQ_{t_{1},t}^{\pi\setminus \{s,t_1\}} f|_\infty e^{\tilde{C}(m,d,\left|V\right|_{\mathcal{C}^m_b})\tilde{\chi}(s,t_1)}\\
			&\leq |f|_{\mathcal{C}^m_b(\mathbb{R}^d,\mathbb{R})}e^{\tilde{C}(m,d,\left|V\right|_{\mathcal{C}^m_b})\left(\tilde{\chi}(t_{1},t)+\tilde{\chi}(s,t_{1})\right)}\\
			&\leq |f|_{\mathcal{C}^m_b(\mathbb{R}^d,\mathbb{R})} e^{\tilde{C}(m,d,\left|V\right|_{\mathcal{C}^m_b})\tilde{\chi}(s,t)}.
		\end{aligned}
	\end{equation}
\end{proof}

We end this section with two small results concerning the strong continuity and contraction property of $(Q_{s,t})_{(s,t)\in \Delta_T}$.

\begin{lemma}\label{vanish}
	For any $(s,t)\in \Delta_T$ and any $f\in \hat{C}(\mathbb{R}^d,\mathbb{R})$, it holds $Q_{s,t}f\in \hat{C}(\mathbb{R}^d,\mathbb{R})$. Furthermore it holds 
	\begin{align*}
		|Q_{s,t}f|_\infty \leq |f|_\infty.
	\end{align*}
\end{lemma}

\begin{proof}
	It is enough to show that 
	\begin{align*}
		|X_t^{s,t,x}|\rightarrow \infty,
	\end{align*}
	almost surely, as $|x|\rightarrow \infty$. This can be shown in a similar way as it is done in Corollary 19.31 in \cite{SchillingPartzsch+2012}. The second part of the result was already proven in Lemma \ref{stability}.
\end{proof}

\begin{corollary}\label{strong_continuity}
	The family $(Q_{s,t})_{(s,t)\in \Delta_T}\subseteq L(\hat{C}(\mathbb{R}^d,\mathbb{R}),\hat{C}(\mathbb{R}^d,\mathbb{R}))$ is strongly continuous, that is for any $f\in \hat{C}(\mathbb{R}^d,\mathbb{R})$ we have 
	\begin{align*}
		|Q_{s,t}f-f|_\infty\rightarrow 0,
	\end{align*}
	as $(s,t)\rightarrow 0$
\end{corollary}

\begin{proof}
	This is an immediate consequence of the uniform continuity of $f$ and Lemma \ref{boundSol}.
\end{proof}

%

\subsection{Commutator bounds}
\label{ssec:commutator}

For fixed \( s<t\in[0,T] \) recall that \( Q_{s,t}\in L(\hat{C}(\mathbb{R}^d,\mathbb{R}))\cap L(\mathcal{C}^{m+2}_b(\mathbb{R}^d,\mathbb{R}))\) is the linear operator \( f\mapsto \E(f(X^{s,t,\cdot}_t)) \) where $X_{\cdot}^{s,t,x}$ is the solution to the random ODE
	\[
	\begin{aligned}
		&dX^{s,t;x}_r= \sum_{i=0}^{d}V_i(X^{s,t;x}_r)dW^{s,t}_r\quad r\in [s,t]
		\\
		&X^{s,t;x}_s=x\,.
	\end{aligned}
	\]
Our next goal is to establish suitable estimates on the 3-parameter quantity \(Q_{s,t}-Q_{s,u}Q_{u,t} \) for \( s<u<t\in [0,T]\), for the operator norm
\( \|Q\|=\sup_{|f|_{\mathcal \mathcal{C}_b^{m+1}}}|Qf|_{\infty}.\)

\begin{theorem}
	\label{commutator}
	 For any $f\in \mathcal{C}^{m+1}_b(\mathbb{R}^d,\mathbb{R})$ it holds
	\begin{align*}
		\left|\E\left[f(X_t^{s,t,x})\right]- \E\left[f(X_t^{u,t,X_u^{s,u,x}})\right] \right|\leq C(m,d,p) \sup_{\substack{(i_1,...,i_k)\in \mathcal{A}_{m}\\\alpha=(i_0,i_1,...,i_k)\notin \mathcal{A}_m}}|V_\alpha f|_\infty  \tilde{\chi}(s,t)^{\frac{m+1}{2}}.
	\end{align*}
	for all $s\leq u\leq t$ in $[0,T]$ and some constant $C(m,d,p)>0$.
\end{theorem}

\begin{proof}
	Using Taylor formula, we get 
	\begin{align*}
		f(X_t^{s,t,x})
		&=\sum_{\alpha\in \mathcal{A}_m} V_\alpha f(x)I^{s,t,\alpha}_{s,t}+R_m(s,t,x,f)\\
		&=\sum_{k=0}^{m}\sum_{\alpha\in \mathcal{A}_k\setminus\mathcal{A}_{k-1}} V_\alpha f(x)I^{s,t,\alpha}_{s,t}+R_m(s,t,x,f),
	\end{align*}
	where 
	\begin{align*}
		R_m(s,t,x,f)=\sum_{\substack{(i_1,...,i_k)\in \mathcal{A}_{m}\\\tilde{\alpha}=(i_0,i_1,...,i_k)\notin \mathcal{A}_m}}^{}I_{s,t}^{\tilde{\alpha}}[W^{s,t}](V_{\tilde{\alpha}} f(X_\cdot^{s,t,x})),
	\end{align*} 
	Similarly we have 
	\begin{align*}
		f(X_t^{u,t,X_u^{s,u,x}})
		&=\sum_{\alpha\in \mathcal{A}_m} V_\alpha f(X_u^{s,u,x})I^{u,t,\alpha}_{u,t}+R_m(u,t,X_u^{s,u,x},f)\\
		&=\sum_{k=0}^{m}\sum_{\alpha\in \mathcal{A}_k\setminus \mathcal{A}_{k-1}} V_\alpha f(X_u^{s,u,x})I^{u,t,\alpha}_{u,t}+R_m(u,t,X_u^{s,u,x},f)
	\end{align*}
	Again, using Taylor formula, we get for $\alpha \in \mathcal{A}_k\setminus\mathcal{A}_{k-1}$ and $\tilde{m}:=m-k$
	\begin{align*}
		V_\alpha f(X_u^{s,u,x})&=\sum_{\tilde{\alpha}\in \mathcal{A}_{\tilde{m}}} V_{\tilde{\alpha}} V_\alpha f(x)I^{s,u,\tilde{\alpha}}_{s,u}+R_{\tilde{m}}(s,u,x,V_{\alpha} f)\\
		&=\sum_{\tilde{\alpha}\in \mathcal{A}_{\tilde{m}}} V_{\alpha*\tilde{\alpha}} f(x)I^{s,u,\tilde{\alpha}}_{s,u}+R_{\tilde{m}}(s,u,x,V_{\alpha} f)\\
		&=\sum_{k_2=0}^{\tilde{m}}\sum_{\tilde{\alpha}\in \mathcal{A}_{k_2}\setminus \mathcal{A}_{k_2-1}} V_{\alpha*\tilde{\alpha}} f(x)I^{s,u,\tilde{\alpha}}_{s,u}+R_{\tilde{m}}(s,u,x,V_{\alpha} f).
	\end{align*}
	Therefore
	\begin{align*}
		f(X_t^{u,t,X_u^{s,u,x}})
		&=\sum_{k_1=0}^{m}\sum_{k_2=0}^{m-k_1}\sum_{\alpha\in \mathcal{A}_{k_1}\setminus\mathcal{A}_{k_1-1}}\sum_{\tilde{\alpha}\in \mathcal{A}_{k_2}\setminus \mathcal{A}_{k_2-1}} V_{\alpha*\tilde{\alpha}} f(x)I^{s,u,\tilde{\alpha}}_{s,u}I^{u,t,\alpha}_{u,t}\\
		&\quad + \sum_{k_1=0}^{m}\sum_{k_2=0}^{m-k_1}\sum_{\alpha\in \mathcal{A}_{k_1}\setminus\mathcal{A}_{k_1-1}}\sum_{\tilde{\alpha}\in \mathcal{A}_{k_2}\setminus \mathcal{A}_{k_2-1}} R_{m-k_1}(s,u,x,V_{\alpha} f)I^{u,t,\alpha}_{u,t}\\
		&\quad +R_m(u,t,X_u^{s,u,x},f).
	\end{align*}
	Substituting $k:=k_1+k_2$ we get 
	\begin{align*}
		f(X_t^{u,t,X_u^{s,u,x}})
		&=\sum_{k=0}^{m}\sum_{i=0}^{k}\sum_{\alpha\in \mathcal{A}_{k-i}\setminus\mathcal{A}_{k-i-1}}\sum_{\tilde{\alpha}\in \mathcal{A}_{i}\setminus \mathcal{A}_{i-1}} V_{\alpha*\tilde{\alpha}} f(x)I^{s,u,\tilde{\alpha}}_{s,u}I^{u,t,\alpha}_{u,t}\\
		&\quad + \sum_{k=0}^{m}\sum_{i=0}^{k}\sum_{\alpha\in \mathcal{A}_{k-i}\setminus\mathcal{A}_{k-i-1}}\sum_{\tilde{\alpha}\in \mathcal{A}_{i}\setminus \mathcal{A}_{i-1}} R_{i+(m-k)}(s,u,x,V_{\alpha} f)I^{u,t,\alpha}_{u,t}\\
		&\quad +R_m(u,t,X_u^{s,u,x},f).
	\end{align*}
	
	By Assumption \ref{a1}
	\begin{align*}
		\E\left[S_{s,t}^m\right]&=\sum_{k=0}^{m}\sum_{\alpha=(i_1,...,i_l)\in \mathcal{A}_k\setminus \mathcal{A}_{k-1}}\mathbb{E}\left[I_{s,t}^{s,t,\alpha}\right]e_{i_1}\otimes ...\otimes e_{i_l}
	\end{align*}
	satisfies 
	\begin{align*}
		&\E\left[S_{s,t}^m\right]=\E\left[S_{s,u}^m\right]\otimes \E\left[S_{u,t}^m\right]\\
		&=\sum_{k=0}^m \sum_{i=0}^{k}\left(\pi_{k-i}(\E\left[S_{s,u}^m\right])\otimes \pi_{i}(\E\left[S_{u,t}^m\right])\right)\\
		&=\sum_{k=0}^{m}\sum_{i=0}^{k}\left(\sum_{\substack{\alpha\in \mathcal{A}_{k-i}\setminus \mathcal{A}_{k-i-1}\\\alpha=(i_1,...,i_{l_1})}}\mathbb{E}\left[I_{s,u}^{s,u,\alpha}\right]e_{i_1}\otimes ...\otimes e_{i_{l_1}}\right) \otimes \left(\sum_{\substack{\tilde{\alpha}\in \mathcal{A}_{i}\setminus \mathcal{A}_{i-1}\\\tilde{\alpha}=(j_1,...,j_{l_2})}}\mathbb{E}\left[I_{u,t}^{u,t,\tilde{\alpha}}\right]e_{j_1}\otimes ...\otimes e_{j_{l_2}}\right)\\
		&=\sum_{k=0}^{m}\sum_{i=0}^{k}\sum_{\substack{\alpha\in \mathcal{A}_{k-i}\setminus \mathcal{A}_{k-i-1}\\\alpha=(i_1,...,i_{l_1})}}\sum_{\substack{\tilde{\alpha}\in \mathcal{A}_{i}\setminus \mathcal{A}_{i-1}\\\tilde{\alpha}=(j_1,...,j_{l_2})}}\mathbb{E}\left[I_{s,u}^{s,u,\alpha}\right]\mathbb{E}\left[I_{u,t}^{u,t,\tilde{\alpha}}\right]e_{i_1}\otimes ...\otimes e_{i_{l_1}}\otimes e_{j_1}\otimes ...\otimes e_{j_{l_2}}.
	\end{align*}
	
	Therefore we get 
	\begin{align*}
		\E\left[f(X_t^{s,t,x})\right]&=\sum_{k=0}^{m}\sum_{i=0}^{k} \sum_{\substack{\hat{\alpha}=(i_1,...,i_{l_1})\in \mathcal{A}_{k-i}\setminus \mathcal{A}_{k-i-1}\\\tilde{\alpha}=(j_1,...,j_{l_2})\in \mathcal{A}_{i}\setminus \mathcal{A}_{i-1}}}\mathbb{E}\left[I_{s,u}^{s,u,\hat{\alpha}}\right]\mathbb{E}\left[I_{u,t}^{u,t,\tilde{\alpha}}\right] V_{\hat{\alpha}*\tilde{\alpha}} f(x)\\
		&\quad +\E\left[R_m(s,t,x,f)\right]
	\end{align*}
	Now by definition $I^{s,u,\tilde{\alpha}}_{s,u},I^{u,t,\alpha}_{u,t}$ are independent, since $W^{s,u}$ and $W^{u,t}$ are independent. This leads to 
	\begin{align*}
		&\left|\E\left[f(X_t^{s,t,x})\right]- \E\left[f(X_t^{u,t,X_u^{s,u,x}})\right] \right|\\
		&\leq  \sum_{k=0}^{m}\sum_{i=0}^{k}\sum_{\alpha\in \mathcal{A}_{k-i}\setminus\mathcal{A}_{k-i-1}}\sum_{\tilde{\alpha}\in \mathcal{A}_{i}\setminus \mathcal{A}_{i-1}} \left|V_{\alpha*\tilde{\alpha}}f\right|_\infty \E\left[\left|R_{i+(m-k)}(s,u,x,V_{\alpha} f)\right|\right]\E\left[\left|I^{u,t,\alpha}_{u,t}\right|\right]\\
		&\quad +\E\left[\left|R_m(u,t,X_u^{s,u,x},f)\right|\right]+\E\left[|R_m(s,t,x,f)|\right]
		\,.
	\end{align*}
	
	Now using \cref{a1} (i), we have
	\begin{equation}\label{computations_R}
	\begin{aligned}
		&\sum_{\substack{(i_1,...,i_k)\in \mathcal{A}_{m}\\\tilde{\alpha}=(i_0,i_1,...,i_k)\notin \mathcal{A}_m}}^{}\left|\int_{s<t_0<...<t_l<t}Y_{t_1}^{i_0,...,i_k}dW_{t_1}^{s,t,i_0}....dW_{t_l}^{s,t,i_k}\right|\\
		&\leq \sup_{\substack{(i_1,...,i_k)\in \mathcal{A}_{m}\\\alpha=(i_0,i_1,...,i_k)\notin \mathcal{A}_m}}|Y^{i_0,...,i_k}|_\infty \sum_{\substack{(i_1,...,i_k)\in \mathcal{A}_{m}\\\tilde{\alpha}=(i_0,i_1,...,i_k)\notin \mathcal{A}_m}}^{}\int_{s<t_0<...<t_l<t}|dW_{t_0}^{s,t,i_0}|....|dW_{t_l}^{s,t,i_k}|\\
		&\leq C \sup_{\substack{(i_1,...,i_k)\in \mathcal{A}_{m}\\\alpha=(i_0,i_1,...,i_k)\notin \mathcal{A}_m}}|Y^{i_0,...,i_k}|_\infty \frac{|\mathcal{A}_{m+2}\setminus \mathcal{A}_m|}{((m+1)/2)!}\tilde{\chi}(s,t)^{\frac{m+1}{2}}
		\,. 
	\end{aligned}
	\end{equation}
	Therefore we get
	\begin{align*}
		\E\left[\left|R_m(u,t,X_u^{s,u,x},f)\right|\right]	&\leq C(m,d,p)\sup_{\substack{(i_1,...,i_k)\in \mathcal{A}_{m}\\\alpha=(i_0,i_1,...,i_k)\notin \mathcal{A}_m}}|V_\alpha f|_\infty  \tilde{\chi}(s,t)^{\frac{m+1}{2}},
	\end{align*} 
	for some $C(m,d,p)>0$. Putting everything together, we end up with 
	\[
		\left|\E\left[f(X_t^{s,t,x})\right]- \E\left[f(X_t^{u,t,X_u^{s,u,x}})\right] \right|\leq C(m,d,p) \sup_{\substack{(i_1,...,i_k)\in \mathcal{A}_{m}\\\alpha=(i_0,i_1,...,i_k)\notin \mathcal{A}_m}}|V_\alpha f|_\infty  \tilde{\chi}(s,t)^{\frac{m+1}{2}}.
		\qedhere
	\]
\end{proof}

\subsection{End of the proof of \cref{thm:main}}

We consider the Banach spaces $(\hat{C}(\mathbb{R}^d,\mathbb{R}),|\cdot|_\infty)$ and $(\hat{C}(\mathbb{R}^d,\mathbb{R})\cap\mathcal{C}_b^{m+1}(\mathbb{R}^d,\mathbb{R}),\|\cdot \|_{\mathcal{C}_b^{m+1}(\mathbb{R}^d,\mathbb{R})})$. By Lemma \ref{stability}, for any $(s,t)\in \Delta_T$, the operator $Q_{s,t}$ is an element of the operator algebra $L(\hat{C}(\mathbb{R}^d,\mathbb{R}))\cap L(\hat{C}(\mathbb{R}^d,\mathbb{R})\cap \mathcal{C}_b^{m+1}(\mathbb{R}^d,\mathbb{R}))$. Furthermore, as a result of Lemma \ref{strong_continuity} and Theorem \ref{commutator}, the family $(Q_{s,t})_{(s,t)\in \Delta_T}$ is strongly continuous and almost multiplicative with control $\tilde{\chi}$ and $z=\frac{m+1}{2}$. Finally, Lemma \ref{growth} and Lemma \ref{vanish} provide the growth rates
\begin{align*}
	\sup_{|f|_\infty \leq1}|Q_{s,t}^\pi f|_\infty\leq 1
\end{align*}
and 
\begin{align*}
	\max_{l=0,...,m+1}\left(\sup_{|f|_{\mathcal{C}_b^l(\mathbb{R}^d,\mathbb{R})}\leq 1}|Q_{s,t}^\pi f|_{\mathcal{C}^l_b(\mathbb{R}^d,\mathbb{R})}\right)\leq e^{\tilde{C}(m,d,\left|V\right|_{\mathcal{C}^{m+1}_b})\tilde{\chi}(s,t)},
\end{align*}
necessary to apply Theorem \ref{sewing}. As a result of Theorem \ref{sewing}, there exists a unique strongly continuous evolution system of contractions $(P_{s,t})_{(s,t)\in \Delta_T}$, such that 
\begin{align*}
	\sup_{\substack{f\in \hat{C}(\mathbb{R}^d,\mathbb{R})\cap \mathcal{C}_b^{m+1}(\mathbb{R}^d,\mathbb{R})\\|f|_{\mathcal{C}_b^{m+1}}\leq 1}} |Q_{s,t}^\pi f-P_{s,t}f|_\infty&\leq  2^{\frac{m+1}{2}}e^{\tilde{C}\tilde{\chi}(s,t)}\zeta(\frac{m+1}{2})\tilde{\chi}(s,t)\max_{1\leq i\leq k}\{\tilde{\chi}^{\frac{m-1}{2}}(t_i,t_{i+1})\}
\end{align*}
and 
\begin{align*}
		\sup_{\substack{f\in \hat{C}(\mathbb{R}^d,\mathbb{R})\cap \mathcal{C}_b^{m+1}(\mathbb{R}^d,\mathbb{R})\\|f|_{\mathcal{C}_b^{m+1}}\leq 1}} |Q_{s,t} f-P_{s,t}f|_\infty&\leq 2^{\frac{m+1}{2}}\zeta(\frac{m+1}{2})e^{\tilde{C}\tilde{\chi}(s,t)}\tilde{\chi}^{\frac{m+1}{2}}(s,t).
\end{align*}
It remains to determine the generator of  $(P_{s,t})_{(s,t)\in \Delta_T}$ on $\mathcal{C}_c^\infty(\mathbb{R}^d,\mathbb{R})$.

\begin{lemma}
	For any $t\in [0,T]$ and $f\in \mathcal{C}_c^\infty(\mathbb{R}^d,\mathbb{R})$ it holds 
	\begin{align*}
		\lim\limits_{h\downarrow 0}|\frac{P_{t,t+h}f-f}{h}-(V_0f+\frac{1}{2}\sum_{i=1}^{d}\sum_{j=1}^{d}a_{ij}(t)V_iV_jf)|_{\infty}\rightarrow 0.
	\end{align*}
\end{lemma}

\begin{proof}
	A simple appliation of Taylor formula leads to 
	\begin{align*}
		f(X_{t+h}^{t,t+h,x})&=f(x)+V_0f(x)h+\sum_{i=1}^{d}V_if(x)I_{t,t+h}^{t,t+h,(i)}+\sum_{i=1}^{d}\sum_{j=1}^{d}V_iV_jf(x)I_{t,t+h}^{t,t+h,(i,j)}\\
		&\quad +R_3(t,t+h,x,f).
	\end{align*}
	Since 
	\begin{align*}
		\mathbb{E}\left[|R_3(t,t+h,x,f)|\right]\leq C \sup_{ \substack{(i_1,...,i_k)\in \mathcal{A}_{3}\\\alpha=(i_0,i_1,...,i_k)\notin \mathcal{A}_3}}|V_\alpha f|\tilde{\chi}(t,t+h)^{2},
	\end{align*}
	Assumption \ref{a1} allows us to conclude, that
	\begin{align*}
		\lim\limits_{h\downarrow 0}|\frac{Q_{t,t+h}f-f}{h}-(V_0f+\frac{1}{2}\sum_{i=1}^{d}\sum_{j=1}^{d}a_{ij}(t)V_iV_jf)|_\infty=0.
	\end{align*}
	Using the triangle inequality and Assumption \ref{a1} again, we end up with 
	\begin{align*}
		&|\frac{P_{t,t+h}f-f}{h}-(V_0f+\frac{1}{2}\sum_{i=1}^{d}\sum_{j=1}^{d}a_{ij}(t)V_iV_jf)|_{\infty}\\
		&\leq |\frac{P_{t,t+h}f-Q_{t,t+h}f}{h}|_{\infty}+|\frac{Q_{t,t+h}f-f}{h}-(V_0f+\frac{1}{2}\sum_{i=1}^{d}\sum_{j=1}^{d}a_{ij}(t)V_iV_jf)|_{\infty}\\
		&\leq |\frac{2^{\frac{m+1}{2}}\zeta(\frac{m+1}{2})e^{\tilde{C}\tilde{\chi}(t,t+h)}\tilde{\chi}(t,t+h)^{\frac{m+1}{2}}}{h}|\\
		&\quad +|\frac{Q_{t,t+h}f-f}{h}-(V_0f+\frac{1}{2}\sum_{i=1}^{d}\sum_{j=1}^{d}a_{ij}(t)V_iV_jf)|_{\infty}\\
		&\rightarrow 0,
	\end{align*}
	as $h\downarrow 0$.
\end{proof}

Since for any $(s,t)\in \Delta_T$ and any partition $\pi=\{s=t_0<t_1<...<t_k=t\}$ the operator $Q_{s,t}^\pi\in L(\hat{C}(\mathbb{R}^d,\mathbb{R}),\hat{C}(\mathbb{R}^d,\mathbb{R}))$ is positivity preserving, $\psi_{s,t}$ is also positivity preserving.
Let now $x\in \mathbb{R}^d$ and $s\in [0,T]$. By Riesz's representation theorem, there exists a family of Borel measures $(p_{s,t}(x,dy))_{t\in [s,T]}$, such that 
\begin{align*}
	P_{s,t}f(x)=\int_{\mathbb{R}^d}^{}f(y)p_{s,t}(x,dy),
\end{align*}
for any $f\in \hat{C}(\mathbb{R}^d,\mathbb{R})$. Thanks to the evolution property and strong continuity, we have for any $f\in \mathcal{D}(\mathcal{A}_t)$
\begin{align*}
	\dfrac{\partial}{\partial t}P_{s,t}f=P_{s,t}\mathcal{A}_tf.
\end{align*}
Therefore it holds for any $t\in [s,T]$ and any $f\in\mathcal{C}_c^\infty (\mathbb{R}^d,\mathbb{R})$
\begin{equation}\label{forward}
	\begin{aligned}
		\frac{\partial}{\partial t}\int_{\mathbb{R}^d}^{}f(y)dp_{s,t}(x,dy)&=\int_{\mathbb{R}^d}^{}\mathcal{A}_tf(y)dp_{s,t}(x,dy),\\
		p_{s,s}(x,dy)&=\delta_x(dy).
	\end{aligned}
\end{equation}
Due to the non-degeneracy assumption, by Theorem 1 of \cite{BentataCont}, the unique solution to equation (\ref{forward}) is given by the conditional distribution
\begin{align*}
	p_{s,t}(x,dy)=\mathcal{L}_{Q}(X_t\in dy|X_s=x),
\end{align*}
where $Q$ is the unique solution to the martingale problem $(x,(\mathcal{A}_t)_{t\in [0,T]},\mathcal{C}_c^\infty(\mathbb{R}^d,\mathbb{R}))$ and $(X_t)_{t\in [0,T]}$ is the canonical process on $\mathcal{C}([0,T],\mathbb{R}^d)$. This finished the proof of our main theorem.

\section{Examples}

In this section we will construct a family of random paths of bounded variation $(W^{s,t})_{(s,t)\in \Delta_T}$ that satisfies Assumption \ref{a1} with order $m=5$, such that for any $f\in \mathcal{C}_c^\infty(\mathbb{R},\mathbb{R})$ and any sequence of partitions $(\pi_n)_n$ of $[0,T]$ with $|\pi_n|\rightarrow 0$, the approximation $Q_{0,T}^{\pi_n}f$ converges to $\mathbb{E}\left[f(X_t)\right]$, where $(X_t)_{t\in [0,T]}$ is the unique solution to
\begin{equation}\label{gaussian}
	\begin{aligned}
		dX_t&= V_0(X_t)dt+V(X_t)\circ dM_t \quad t\in [0,T]
		\\
		X_0&=x,
	\end{aligned}
\end{equation}
for some real valued Gaussian martingale $(M_t)_{t\in [0,T]}$.

\begin{lemma}\label{expected_signature_mart}
	Let $M$ be a real valued, continuous Gaussian martingale with quadratic variation $(\langle M \rangle_r)_{r\in [0,T]}=(\E\left[M_r^2\right])_{r\in [0,T]}$. Let $(s,t)\in \Delta_T$ and $\alpha\in \mathcal{A}_5$, then it holds
	\begin{align*}
		\E\left[I_{s,t}^\alpha[\circ M]\right]&=\begin{cases}
			0, & \text{if } \|\alpha\|\text{ is odd, or }\alpha=(1,0,1),\\
			\frac{1}{k!}(t-s)^k, & \text{if }\alpha \in \{(i_1,...,i_k)|i_1=...=i_k=1\},\\
			\frac{1}{k}I_{s,t}^{(i_1,...,i_{k/2})}[\langle M\rangle], & \text{if }\alpha \in \{(i_1,...,i_k)|i_1=...=i_k=0\},\text{ for }k=2,4,\\
			\frac{1}{2}I_{s,t}^{(0,1)}[\langle M\rangle], & \text{if }\alpha=(0,1,1),\\
			\frac{1}{2}I_{s,t}^{(1,0)}[\langle M\rangle], & \text{if }\alpha=(1,1,0).
		\end{cases}
	\end{align*}
\end{lemma}

\begin{proof}
	This follows by explicitly calculating $\E\left[I_{s,t}^\alpha[\circ M]\right]$ for any $\alpha\in \mathcal{A}_5$.
\end{proof}

\begin{definition}
	For $f,g\in \mathcal{C}^{1-\text{var}}([0,T],\mathbb{R})$ we introduce the notation of the commutator 
	\begin{align*}
		[f,g]_{s,t}:=\int_{s}^{t}\int_{s}^{u}df(r)dg(u)-\int_{s}^{t}\int_{s}^{u}dg(r)df(u),
	\end{align*} 
	for $0\leq s\leq t\leq T$.
\end{definition}

\begin{lemma}\label{example}
	Let $M$ be a real valued, continuous Gaussian martingale with quadratic variation $(\langle M \rangle_r)_{r\in [0,T]}=(\E\left[M_r^2\right])_{r\in [0,T]}$. Let $h:\mathbb{R}\rightarrow \mathbb{R},r\mapsto r$. Assume that for any $(s,t)\in \Delta_T$ it holds
	\begin{align*}
		\frac{135}{88}[\langle M\rangle,h]_{s,t}^2\leq (t-s)^2(\langle M \rangle_t-\langle M\rangle_s)^2.
	\end{align*}
	For $(s,t)\in \Delta_T$ we define $\omega^{s,t}:[0,1]\rightarrow \mathbb{R}$ as
	\begin{align*}
		\omega^{s,t}_r:=\begin{cases}
			a(s,t)r & r\in [0,\frac{1}{3}]\\
			b(s,t)r+b_0(s,t)& r\in [\frac{1}{3},\frac{2}{3}]\\
			c(s,t)r+c_0(s,t)& r\in [\frac{2}{3},1]
		\end{cases},
	\end{align*}
	where 
	\begin{align*}
		a(s,t)&:=-\frac{\kappa(s,t) +\frac{1}{2}[\langle M\rangle,h]_{s,t}-8I_{s,t}^{(1,0)}(\langle M\rangle)}{2(t-s)\sqrt{\langle M\rangle_t-\langle M\rangle_s}},\\
		b(s,t)&:=\frac{\kappa(s,t) +[\langle M\rangle,h]_{s,t}-2I_{s,t}^{(1,0)}(\langle M\rangle)}{(t-s)\sqrt{\langle M\rangle_t-\langle M\rangle_s}},\\
		c(s,t)&:=-\frac{\kappa(s,t) +\frac{17}{2}[\langle M\rangle,h]_{s,t}-8I_{s,t}^{(1,0)}(\langle M\rangle)}{2(t-s)\sqrt{\langle M\rangle_t-\langle M\rangle_s}},\\
		b_0(s,t)&:=-\frac{\kappa(s,t) +\frac{1}{2}[\langle M\rangle,h]_{s,t}-4I_{s,t}^{(1,0)}(\langle M\rangle)}{2(t-s)\sqrt{\langle M\rangle_t-\langle M\rangle_s}},\\
		c_0(s,t)&:=\frac{\kappa(s,t) +\frac{13}{2}[\langle M\rangle,h]_{s,t}-4I_{s,t}^{(1,0)}(\langle M\rangle)}{2(t-s)\sqrt{\langle M\rangle_t-\langle M\rangle_s}},
	\end{align*}
	for
	\begin{align*}
		\kappa(s,t):=\sqrt{22}\sqrt{(t-s)^2(\langle M \rangle_t-\langle M\rangle_s)^2-\frac{135}{88}[\langle M\rangle,h]_{s,t}^2}.
	\end{align*}
	Then any family $(W^{(s,t)})_{(s,t)\in \Delta_T}$ with 
	\begin{equation}\label{cubFam}
		\begin{aligned}
		\mathbb{P}\left[W^{s,t}=\sqrt{3}\omega^{s,t}\circ \phi\right]&=\frac{1}{6}\\
		\mathbb{P}\left[W^{s,t}=-\sqrt{3}\omega^{s,t}\circ \phi\right]&=\frac{1}{6}\\
		\mathbb{P}\left[W^{s,t}_r=0\right]&=\frac{2}{3},
		\end{aligned}
	\end{equation}
	where $\phi:[s,t]\rightarrow [0,1],r\mapsto \frac{r-s}{t-s}$, satisfies 
	\begin{align*}
		\E\left[S_{s,t}^5[W^{s,t}]\right]=\E\left[S_{s,t}^5[\circ M]\right],
	\end{align*}
	for any $(s,t)\in \Delta_T$.
\end{lemma}

\begin{remark}
	If $M$ is just a one dimensional Brownian motion, then $\langle M\rangle_r=r$ and
	\begin{align*}
		I_{s,t}^{(1,0)}(\langle M\rangle)&=\frac{1}{2}(t-s)^2\\
		[\langle M\rangle,h]_{s,t}&=0,
	\end{align*}
	for all $(s,t)\in \Delta_T$. The expression for $\omega^{s,t}$ then simplifies to 
	\begin{align*}
		\omega^{s,t}_r=\begin{cases}
			\frac{1}{2}(4-\sqrt{22})r & r\in [0,\frac{1}{3}]\\
			(\sqrt{22}-1)r+\frac{1}{2}(2-\sqrt{22})& r\in [\frac{1}{3},\frac{2}{3}]\\
			\frac{1}{2}(4-\sqrt{22})r+\frac{1}{2}(\sqrt{22}-2)& r\in [\frac{2}{3},1]
		\end{cases}.
	\end{align*}
	Therefore the approximation from Lemma \ref{example} coincides in this case with the cubature formula constructed in \cite{article}.
\end{remark}

\begin{proof}
	The result can be proven by simply computing the corresponding expected iterated integrals of a family $(W^{(s,t)})_{(s,t)\in \Delta_T}$ which satisfies (\ref{cubFam}). However, we will explain how to derive the coefficients for
	\begin{align*}
		\mathbb{P}\left[W^{s,t}=\omega^{s,t}\circ \phi\right]&=\lambda_1\\
		\mathbb{P}\left[W^{s,t}=-\omega^{s,t}\circ \phi\right]&=\lambda_1\\
		\mathbb{P}\left[W^{s,t}_r=0\right]&=\lambda_2,
	\end{align*}
	where $\phi(r)=\frac{r-s}{t-s}$ and
	\begin{align*}
		\tilde{\omega}^{s,t}_r:=\begin{cases}
			a(s,t)r & r\in [0,\frac{1}{3}]\\
			b(s,t)r+b_0(s,t)& r\in [\frac{1}{3},\frac{2}{3}]\\
			c(s,t)r+c_0(s,t)& r\in [\frac{2}{3},1].
		\end{cases}
	\end{align*}
	For any $(s,t)\in \Delta_T$ we have
	\begin{align*}
		\int_{s}^{t}\left(\omega^{s,t}_r\right)^2dr&=\frac{(t-s)}{81}\left(a(s,t)^2+7b(s,t)^2+19c(s,t)^2\right)\\
		&\quad + \frac{(t-s)}{3}\left(b(s,t)b_0(s,t)+b_0(s,t)^2+\frac{5}{3}c(s,t)c_0(s,t)+c_0(s,t)^2\right)
	\end{align*}
	and
	\begin{align*}\label{int}
		\int_{s}^{t}\omega^{s,t}_rdr&=\frac{(t-s)}{18}\left(a(s,t)+3b(s,t)+5c(s,t)\right)\\
		&\quad + \frac{(t-s)}{3}\left(b_0(s,t)+c_0(s,t)\right).
	\end{align*}
	Following a similar pattern as in \cite{PASSEGGERI20201226}, in order to match the expectation of the iterated integrals of $W^{s,t}$ and $M$, we only need to take the following equations into consideration:
	\begin{equation}\label{eq1}
		2\lambda_1\int_{s}^{t}(\omega_r^{s,t})^2dr=\int_{s}^{t}\left(\langle M\rangle_r-\langle M\rangle_s\right)dr,
	\end{equation}
	\begin{equation}\label{eq2}
	2\lambda_1\left(\frac{1}{2} \int_{s}^{t}(\omega_r^{s,t})^2dr+\int_{s}^{t}\int_{s}^{t_2}\int_{s}^{t_1}drd\omega_{t_1}^{s,t}d\omega_{t_2}^{s,t}\right)=\frac{1}{2}(t-s)\left(\langle M\rangle_t-\langle M\rangle_s\right),
	\end{equation}
	\begin{equation}\label{eq3}
		2\lambda_1 \int_{s}^{t}\int_{s}^{t_2}\int_{s}^{t_1}d\omega^{s,t}_rdt_1d\omega^{s,t}_{t_2}=0,
	\end{equation}
	\begin{equation}\label{eq4}
		2\lambda_1 \left(\omega_t^{s,t}\right)^2=\left(\langle M\rangle_t-\langle M\rangle_s\right),
	\end{equation}
	\begin{equation}\label{eq5}
		2\lambda_1 \left(\omega_t^{s,t}\right)^4=3\left(\langle M\rangle_t-\langle M\rangle_s\right)^2,
	\end{equation}
	\begin{equation}\label{eq6}
		2\lambda_1+\lambda_2=1.
	\end{equation}
	For $\omega^{s,t}$ to be continuous, we get the additional equations
	\begin{align*}
		b_0(s,t)&=(\sqrt{3(\langle M\rangle_t-\langle M\rangle_s)}-\frac{c(s,t)}{3}-\frac{2b(s,t)}{3})\\
		c_0(s,t)&=(\sqrt{3(\langle M\rangle_t-\langle M\rangle_s)}-c(s,t)).
	\end{align*}
	From equation (\ref{eq4}) and (\ref{eq5}) we get $\lambda_1=\frac{1}{6},\lambda_2=\frac{2}{3}$. Now integration by parts yield
	\begin{align*}
		\int_{s}^{t}\int_{s}^{t_2}\int_{s}^{t_1}drd\omega_{t_1}^{s,t}d\omega_{t_2}^{s,t}&=\int_{s}^{t}\int_{s}^{t_2}\omega_{t_2}^{s,t}dt_1d\omega_{t_2}^{s,t}-\int_{s}^{t}\int_{s}^{t_2}\omega_{t_1}^{s,t}dt_1d\omega_{t_2}^{s,t}\\
		&=\int_{s}^{t}\int_{s}^{t_2}(\omega_{t_2}^{s,t}-\omega_{t_1}^{s,t})dt_1d\omega_{t_2}^{s,t}.
	\end{align*}
	Using Fubini's theorem and integration by parts again, we end up with
	\begin{align*}
		\int_{s}^{t}\int_{s}^{t_2}\int_{s}^{t_1}drd\omega_{t_1}^{s,t}d\omega_{t_2}^{s,t}&=\int_{s}^{t}\int_{t_1}^{t}(\omega_{t_2}^{s,t}-\omega_{t_1}^{s,t})d\omega_{t_2}^{s,t}dt_1\\
		&=(t-s)\frac{(\omega_{t}^{s,t})^2}{2}+\int_{s}^{t}\frac{(\omega_{r}^{s,t})^2}{2}dr-\int_{s}^{t}\omega_{r}^{s,t}dr\omega_{t}^{s,t}.
	\end{align*}
	Therefore, equation (\ref{eq2}) can be rewritten as
	\begin{align*}
		2\lambda_1 \left(\frac{\left(\omega_t^{s,t}\right)}{2}^2(t-s)+\int_{s}^{t}(\omega_r^{s,t})^2dr-\int_{s}^{t}\omega_r^{s,t}dr\omega_t^{s,t}\right)&=\frac{1}{2}(t-s)\left(\langle M\rangle_t-\langle M\rangle_s\right).
	\end{align*}
	From equation (\ref{eq1}) and (\ref{eq4}), we get
	\begin{align*}
		\frac{1}{2}(t-s)\left(\langle M\rangle_t-\langle M\rangle_s\right)+\int_{s}^{t}(\langle M\rangle_r-\langle M\rangle_s)dr-2\lambda_1 \int_{s}^{t}\omega_r^{s,t}dr\omega_t^{s,t}&=\frac{1}{2}(t-s)\left(\langle M\rangle_t-\langle M\rangle_s\right).
	\end{align*}
	This leads to 
	\begin{align*}
		\int_{s}^{t}\omega_r^{s,t}dr\omega_t^{s,t}&=3\int_{s}^{t}(\langle M\rangle_r-\langle M\rangle_s)dr.
	\end{align*}
	Thus
	\begin{align*}
		&\sqrt{3(\langle M\rangle_t-\langle M\rangle_s)}\left(\frac{(t-s)}{18}\left(a(s,t)+3b(s,t)+5c(s,t)\right)
		+ \frac{(t-s)}{3}\left(b_0(s,t)+c_0(s,t)\right)\right)\\
		&=3\int_{s}^{t}(\langle M\rangle_r-\langle M\rangle_s).
	\end{align*}
	This yields
	\begin{align*}
		b(s,t)&=a(s,t)-3c(s,t)+\frac{3\sqrt{3}(2(\langle M\rangle _t-\langle M\rangle_s)\sqrt{\langle M\rangle_t-\langle M\rangle_s}-3\frac{1}{t-s}\int_{s}^{t}(\langle M\rangle_r-\langle M\rangle_s)dr)}{\sqrt{\langle M\rangle_t-\langle M\rangle_s}}.
	\end{align*}
	Now again, by the continuity condition for $b_0$, we get 
	\begin{align*}
		a(s,t)=3\sqrt{3(\langle M\rangle_t-\langle M\rangle_s)}-c(s,t)-b(s,t).
	\end{align*}
	Therefore
	\begin{align*}
		b(s,t)=\frac{3\sqrt{3}(5(\langle M\rangle_t-\langle M\rangle_s)-6\frac{1}{t-s}\int_{s}^{t}(\langle M\rangle_r-\langle M \rangle_s)dr)}{2\sqrt{\langle M\rangle_t-\langle M\rangle_s}}-2c(s,t).
	\end{align*}
	Thus
	\begin{align*}
		a(s,t)=c_1(s,t)-\frac{3\sqrt{3}(3(\langle M\rangle_t-\langle M\rangle_s)-6\frac{1}{t-s}\int_{s}^{t}(\langle M\rangle_r-\langle M\rangle_s)dr)}{2\sqrt{\langle M\rangle_t-\langle M\rangle_s}}.
	\end{align*}
	Rewriting the terms for $a$ and $b$, we get
	\begin{align*}
		a(s,t)=c(s,t)+9\sqrt{3}\left(\frac{\int_{s}^{t}\int_{s}^{u}d\langle M\rangle_rdu-\int_{s}^{t}\int_{s}^{u}drd\langle M\rangle_u}{2(t-s)\sqrt{\langle M\rangle_t-\langle M\rangle_s}}\right)
	\end{align*}
	and 
	\begin{align*}
		b(s,t)=3\sqrt{3\left(\langle M\rangle_t-\langle M \rangle_s\right)}-9\sqrt{3}\left(\frac{\int_{s}^{t}\int_{s}^{u}d\langle M\rangle_rdu-\int_{s}^{t}\int_{s}^{u}drd\langle M\rangle_u}{2(t-s)\sqrt{\langle M\rangle_t-\langle M\rangle_s}}\right)-2c(s,t).
	\end{align*}
	Now 
	\begin{align*}
		b_0(s,t)&=\sqrt{3(\langle M\rangle_t-\langle M\rangle_s)}-\frac{c(s,t)}{3}+6\sqrt{3}\left(\frac{\int_{s}^{t}\int_{s}^{u}d\langle M\rangle_rdu-\int_{s}^{t}\int_{s}^{u}drd\langle M\rangle_u}{2(t-s)\sqrt{\langle M\rangle_t-\langle M\rangle_s}}\right)\\
		&\quad -2\sqrt{3\left(\langle M\rangle_t-\langle M \rangle_s\right)}+\frac{4}{3}c(s,t)\\
		&=6\sqrt{3}\left(\frac{\int_{s}^{t}\int_{s}^{u}d\langle M\rangle_rdu-\int_{s}^{t}\int_{s}^{u}drd\langle M\rangle_u}{2(t-s)\sqrt{\langle M\rangle_t-\langle M\rangle_s}}\right) +c(s,t)-\sqrt{3\left(\langle M\rangle_t-\langle M \rangle_s\right)}
	\end{align*}
	and 
	\begin{align*}
		c_0(s,t)&=(\sqrt{3(\langle M\rangle_t-\langle M\rangle_s)}-c(s,t)).
	\end{align*}
	Inserting everything into equation (\ref{eq1}) and solving with respect to $c$ yields
	\begin{align*}
		c(s,t)
		&=-\frac{\kappa(s,t)+\frac{17}{2}[\langle M\rangle,Id]_{s,t}-8I_{s,t}^{(1,0)}(\langle M\rangle)}{2(t-s)\sqrt{\langle M\rangle_t-\langle M\rangle_s}}.
	\end{align*}
	The other unknowns can now be determined by simply inserting the representation for $c$ into the corresponding expression. It is easy to check, that equation (\ref{eq3}) is then automatically satisfied. 
\end{proof}
It is not difficult to see that the path $\omega^{s,t}$ in Lemma \ref{example} is Lipschitz-continuous and satisfies
\begin{align*}
	\|\omega^{s,t}\|_{1-\text{Höl};[s,t]}\leq C (\langle M\rangle_t-\langle M\rangle_s),
\end{align*}
for some $C>0$. Therefore, the  family  $(W^{s,t})_{(s,t)\in \Delta_T}$ given in Lemma \ref{example} satisfies the first part of Assumption \ref{a1} with the control
\begin{align*}
	\chi(s,t):=\langle M\rangle_t-\langle M\rangle_s.
\end{align*} 
The expected signature of the random paths $W^{s,t}$ coincides up to order $5$ with the expected signature of the corresponding Gaussian martingale. As already mentioned in Remark \ref{cubature_remark}, the second part of our Assumption \ref{a1} is already satisfied if the expected signature of the noise approximation coincides up to a certain order $m\geq 2$ with the expected signature of a square-integrable continuous martingale, hence the family  $(W^{s,t})_{(s,t)\in \Delta_T}$ satisfies the second part of Assumption \ref{a1}. Furthermore $W^{s,t}$ is a centered random path for any $(s,t)\in \Delta_T$. We will now verify the last part of our assumption for a particular example.
\begin{example}
	Let $(B_t)_{t\in [0,T]}$ be a one-dimensional Brownian motion. We consider the Gaussian martingale 
	\begin{align*}
		M_t:=\int_{0}^{t}sdB_s.
	\end{align*}
	Then we have $\langle M\rangle_t=\frac{1}{3}t^3$ and $[\langle M\rangle,h]_{s,t}=\frac{1}{6}(s^4-2s^3t+2st^3-t^4)$, hence 
	\begin{align*}
		\frac{135}{88}[\langle M\rangle,h]_{s,t}^2\leq (t-s)^2(\langle M \rangle_t-\langle M\rangle_s)^2.
	\end{align*}
	If we consider the approximation $(W^{s,t})_{(s,t)\in \Delta_T}$ defined as in Lemma \ref{example}, then it holds
	\begin{align*}
		\lim\limits_{h\downarrow 0}\frac{1}{h}\E\left[I_{t,t+h}^{t,t+h,(1,1)}\right]=t^2.
	\end{align*}
	Now $(W^{s,t})_{(s,t)\in \Delta_T}$ satisfies Assumption \ref{a1} and therefore by Theorem \ref{thm:main}, if the coefficients $V_0,V$ are non-degenerate, it holds for any $f\in \mathcal{C}_c^\infty(\mathbb{R},\mathbb{R})$ and any sequence of partitions $(\pi_n)_{n\in \mathbb{N}}$ of $[s,t]$ with $|\pi_n|\rightarrow 0$, that
	\begin{align*}
		Q_{s,t}^{\pi_n}f(x)\rightarrow \E\left[f(X_t^x)\right],
	\end{align*}
	as $n\rightarrow \infty$, where $X^x$ is the solution to 
	\begin{equation}
		\begin{aligned}
			dX_t&= V_0(X_t)dt+V(X_t)\circ dM_t \quad t\in [0,T]
			\\
			X_0&=x.
		\end{aligned}
	\end{equation}

	\begin{figure}[H]
		\centering
		\includegraphics[width=0.4\linewidth]{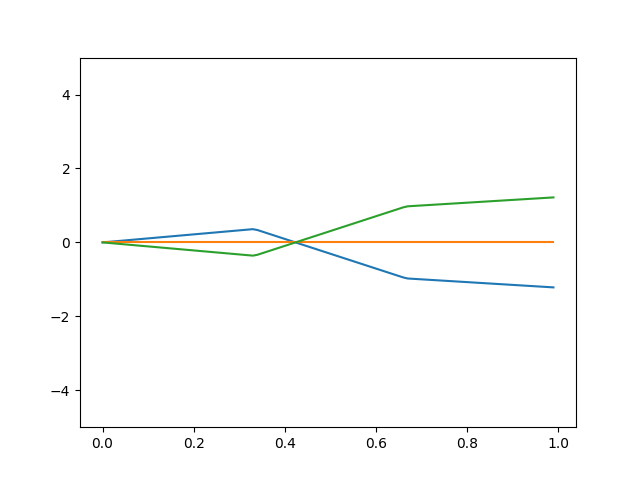}
		\includegraphics[width=0.4\linewidth]{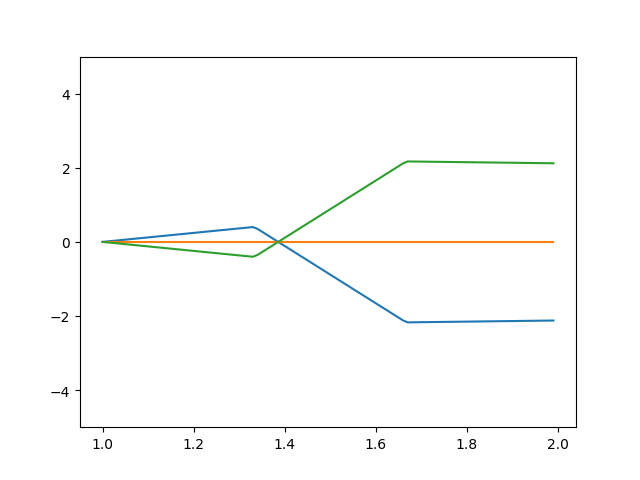}\\
		\includegraphics[width=0.4\linewidth]{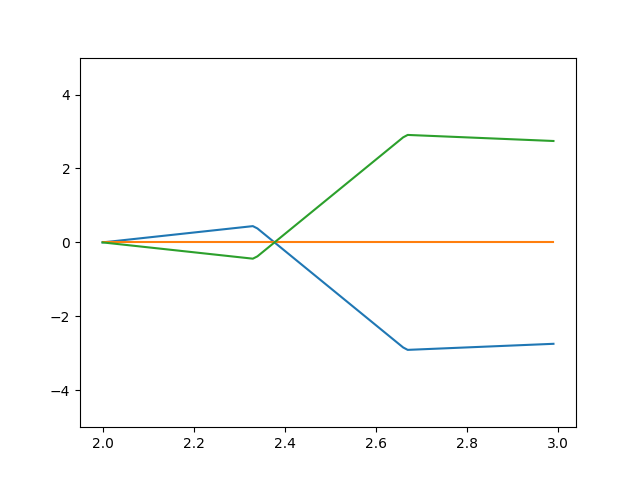}
		\includegraphics[width=0.4\linewidth]{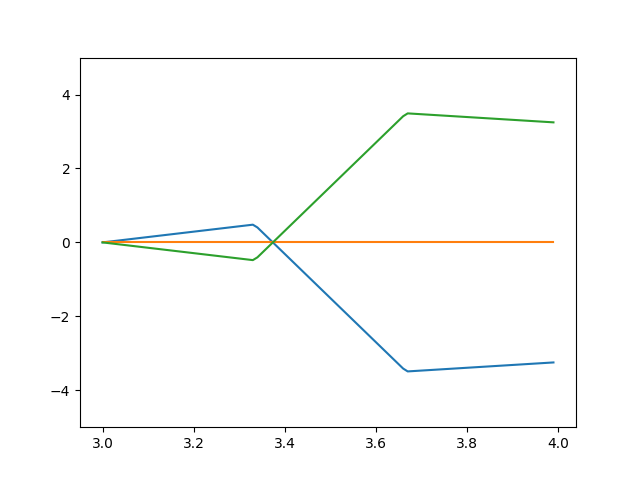}
		\caption{Realizations of the random paths $W^{0,1},W^{1,2},W^{2,3} \text{ and } W^{3,4}$}
	\end{figure}
\end{example}

\appendix
\section{Proof of the Multiplicative Sewing Lemma}

\begin{proof}
	Let $(s,t)\in \Delta_T$ and $\pi=\{s=t_0<t_1<...<t_k=t\}$ be a partition of $[s,t]$. Since $\chi$ is a control, there exists a $0<l<k$, such that 
	\begin{align*}
		\chi(t_{l-1},t_{l+1})\leq  \frac{2\chi(s,t)}{k}.
	\end{align*} 
	We denote by $\hat{\pi}$ the partition of $[s,t]$ obtained by deleting the point $t_l$ from $\pi$. Then
	\begin{align*}
		&\|\mu_{s,t}^{\hat{\pi}}-\mu_{s,t}^\pi\|_{L(X,Y)}\\
		&\leq \|\prod_{i=0}^{l-2}\mu_{t_i,t_{i+1}}\|_{L( X,X)}\|\mu_{t_{l-1},t_{l+1}}-\mu_{t_{l-1},t_l}\mu_{t_{l},t_{l+1}}\|_{L(X,Y)}\|\prod_{i=l+1}^{k}\mu_{t_i,t_{i+1}}\|_{L(Y,Y)}\\
		&\leq \epsilon_1(t_{l+1},t)\chi^z(t_{l-1},t_{l+1})\epsilon_2(s,t_{l-1})\\
		&\leq \epsilon_1(s,t)\epsilon_2(s,t)2^z\chi^z(s,t)k^{-z}
	\end{align*}
	Repeating this procedure until we arrive at the trivial partition $\pi_0=\{s,t\}$, we obtain
	\begin{align}\label{maxin}
		\sup_{\pi}\|\mu_{s,t}-\mu_{s,t}^\pi\|_{L(X,Y)}\leq 2^z\zeta(z)\epsilon_1(s,t)\epsilon_2(s,t)\chi^z(s,t).
	\end{align}
	
	We will now show, that
	
	\begin{align}\label{cin}
		\sup_{|\pi|,|\hat{\pi}|<\epsilon}\|\mu_{s,t}^\pi-\mu_{s,t}^{\hat{\pi}}\|_{L( X,Y)}\rightarrow 0.
	\end{align}
	
	Without loss of generality, let $\tilde{\pi}$ be a refinement of $\pi$. For $\pi=\{s=t_0<t_1<...<t_k=t\}$ we define 
	\begin{align*}
		M_{s,t}^l:=\prod_{i=0}^{l-1}\mu_{t_i,t_{i+1}}\prod_{i=l}^{k-1}\mu_{t_i,t_{i+1}}^{\tilde{\pi}\cap [t_i,t_{i+1}]}.
	\end{align*}
	Now $M_{s,t}^k=\mu_{s,t}^\pi$ and $M_{s,t}^0=\mu_{s,t}^{\tilde{\pi}}$, hence 
	\begin{align*}
		\|\mu_{s,t}^\pi-\mu_{s,t}^{\tilde{\pi}}\|_{L(X,Y)}&\leq \sum_{i=0}^{k}\|M_{s,t}^i-M_{s,t}^{i+1}\|_{L(X,Y)}\\
		&\leq \epsilon_1^2(s,t)\epsilon_2^2(s,t)\sum_{i=0}^{k-1}\|\mu_{t_i,t_{i+1}}-\mu_{t_i,t_{i+1}}^{\tilde{\pi}\cap [t_i,t_{i+1}]}\|_{L(X,Y)}\\
		&\leq 2^z\epsilon_1^2(s,t)\epsilon_2^2(s,t)\zeta(z)\sum_{i=0}^{k-1}\chi^z(t_i,t_{i+1})\\
		&\leq 2^z\epsilon_1^2(s,t)\epsilon_2^2(s,t)\zeta(z)\chi(s,t)\max_{1\leq i\leq k}\{\chi^{z-1}(t_i,t_{i+1})\}\rightarrow 0,
	\end{align*}
	as $|\pi|\rightarrow 0$.\\\\
	Now for any sequence of partitions $\pi_n$ with $|\pi_n|\rightarrow 0$, the sequence $\mu_{s,t}^{\pi_n}$ is Cauchy with respect to $\|\cdot \|_{L(X,Y)}$. Due to the completeness of $L(X,Y)$, there exists a unique $\phi_{s,t}\in L(X,Y)$, independent of the sequence of partitions $(\pi_n)_n$, such that 
	\begin{align*}
		\|\mu_{s,t}^{\pi_n}-\phi_{s,t}\|_{L(X,Y)}\rightarrow 0.
	\end{align*}
	Since 
	\begin{align*}
		\sup_{\pi}\|\mu^{\pi}_{s,t}\|_{L(Y,Y)}\leq \epsilon_2(s,t),
	\end{align*}
	it holds for any $x\in X$
	\begin{align*}
		\|\psi_{s,t}x\|_Y&\leq \|\psi_{s,t}x-\mu_{s,t}^{\pi_n}x\|_Y+\|\mu_{s,t}^{\pi_n}x\|_Y\\
		&\leq \|\psi_{s,t}-\mu_{s,t}^{\pi_n}\|_{L(X,Y)}\|x\|_X+\sup_{\pi}\|\mu_{s,t}^{\pi}\|_{L(Y,Y)}\|x\|_Y\\
		&\rightarrow \sup_{\pi}\|\mu_{s,t}^{\pi}\|_{L(Y,Y)}\|x\|_Y\\
		&\leq \epsilon_2(s,t)\|x\|_Y.
	\end{align*} 
	Therefore, $\phi$ is continuous from $(X,\|\cdot\|_Y) $ to $(Y,\|\cdot\|_Y)$. Since $X\subseteq Y$ densely, there exists a unique extension $\psi\in L(Y,Y)$, such that 
	\begin{align*}
		\|\psi_{s,t}\|_{L(Y,Y)}\leq \epsilon_2(s,t).
	\end{align*} 
	In order to show the evolution property, we observe that for any $x\in X$ it holds
	\begin{align*}
		\|\mu_{s,u}^{\pi_n}\mu_{u,t}^{\pi_n}x-\psi_{s,u}\psi_{u,t}x\|_Y&\leq \|\mu_{s,u}^{\pi_n}\mu_{u,t}^{\pi_n}x-\psi_{s,u}\mu_{u,t}^{\pi_n}x\|_Y\\
		&\quad +\|\psi_{s,u}\mu_{u,t}^{\pi_n}x-\psi_{s,u}\psi_{u,t}x\|_Y\\
		&\leq \|\mu_{s,u}^{\pi_n}\mu_{u,t}^{\pi_n}-\psi_{s,u}\mu_{u,t}^{\pi_n}\|_{L(X,Y)}\|x\|_X\\
		&\quad +\|\psi_{s,u}\mu_{u,t}^{\pi_n}-\psi_{s,u}\psi_{u,t}\|_{L(X,Y)}\|x\|_X\\
		&\leq  \|\mu_{s,u}^{\pi_n}-\phi_{s,u}\|_{L(X,Y)} \|\mu_{u,t}^{\pi_n}\|_{L(X,X)}\|x\|_X\\ &\quad +\|\psi_{s,u}\|_{L(Y,Y)}\|\mu_{u,t}^{\pi_n}-\phi_{u,t}\|_{L(X,Y)}\|x\|_X\\
		&\leq  \|\mu_{s,u}^{\pi_n}-\phi_{s,u}\|_{L(X,Y)} \epsilon_1(s,t)\|x\|_X\\ &\quad +\epsilon_2(s,t)\|\mu_{u,t}^{\pi_n}-\phi_{u,t}\|_{L(X,Y)}\|x\|_X.
	\end{align*}
	Hence 
	\begin{align*}
		\|\mu_{s,u}^{\pi_n}\mu_{u,t}^{\pi_n}x-\psi_{s,u}\psi_{u,t}x\|_Y\rightarrow 0,
	\end{align*}
	for any sequence of partitions $(\pi_n)_n$ with $|\pi_n|\rightarrow 0$. Therefore it holds for any $x\in X$
	\begin{align*}
		\|(\psi_{s,t}-\psi_{s,u}\psi_{u,t})x\|_Y&=\lim\limits_{n\rightarrow \infty} \|(\mu_{s,t}^{\pi_n}-\mu_{s,u}^{\pi_n}\mu_{u,t}^{\pi_n})x\|_Y\\
		&\leq \lim\limits_{n\rightarrow \infty}\|\mu_{s,t}^{\pi_n}-\mu_{s,u}^{\pi_n}\mu_{u,t}^{\pi_n}\|_{L(X,Y)}\|x\|_X\\
		&=0.
	\end{align*}
	For any $y\in Y$, there is a sequence $(x_n)_n\subseteq X$, such that $\|x_n-y\|_Y\rightarrow 0$. Thanks to the continuity of $\psi$, we get
	\begin{align*}
		(\psi_{s,t}-\psi_{s,u}\psi_{u,t})y&=\lim\limits_{n\rightarrow \infty}(\psi_{s,t}-\psi_{s,u}\psi_{u,t})x_n\\
		&=0.
	\end{align*}
	Thus, $\psi$ is multiplicative. Furthermore we have 
	\begin{align*}
		\|\mu_{s,t}-\psi_{s,t}\|_{L(X,Y)}&\leq \sup_{\pi}\|\mu_{s,t}-\mu_{s,t}^{\pi}\|_{L(X,Y)}+\|\mu_{s,t}^{\pi_n}-\psi_{s,t}\|_{L(X,Y)}\\
		&\leq 2^z\zeta(z)\epsilon_1(s,t)\epsilon_2(s,t)\omega(s,t)^z + \|\mu_{s,t}^{\pi_n}-\psi_{s,t}\|_{L(X,Y)}\\
		&\rightarrow 2^z\zeta(z)\epsilon_1(s,t)\epsilon_2(s,t)\omega(s,t)^z.
	\end{align*}
	It remains to show that $\psi$ is strongly continuous. For any $x\in X$ it holds
	\begin{align*}
		\|\psi_{s,t}x-x\|_Y&\leq \|\psi_{s,t}x-\mu_{s,t}x\|_Y+\|\mu_{s,t}x-x\|_Y\\
		&\leq \|\psi_{s,t}-\mu_{s,t}\|_{L(X,Y)}\|x\|_X+\|\mu_{s,t}x-x\|_Y\\
		&\leq 2^z\zeta(z)\epsilon_1(s,t)\epsilon_2(s,t)\omega(s,t)^z\|x\|_X+\|\mu_{s,t}x-x\|_Y\\
		&\rightarrow 0,
	\end{align*}
	for $(s,t)\rightarrow 0$. Now $(\psi_{s,t})_{(s,t)\in \Delta_T}$ is strongly continuous on the dense subset $X\subseteq Y$ and 
	\begin{align*}
		\|\psi_{s,t}\|_{L(Y,Y)}\leq \sup_{\pi}\|\mu^{\pi}_{s,t}\|_{L(Y,Y)}\leq \epsilon_1(s,t),
	\end{align*}
	hence $(\psi_{s,t})_{(s,t)\in \Delta_T}$ is strongly continuous on $Y$.\\\\
	Finally we have for any partition $\pi=\{s=t_0<t_1<...<t_k=t\}$ 
	\begin{align*}
		\|\mu_{s,t}^\pi-\psi_{s,t}\|_{L(X,Y)}&\leq \|\mu_{s,t}^\pi-\mu_{s,t}^{\tilde{\pi}}\|_{L(X,Y)}+\|\mu_{s,t}^{\tilde{\pi}}-\psi_{s,t}\|_{L(X,Y)}\\
		&\leq \sup_{\hat{\pi}}\|\mu_{s,t}^\pi-\mu_{s,t}^{\hat{\pi}}\|_{L(X,Y)}+\|\mu_{s,t}^{\tilde{\pi}}-\psi_{s,t}\|_{L(X,Y)}\\
		& \leq 2^z\epsilon_1^2(s,t)\epsilon_2^2(s,t)\zeta(z)\chi(s,t)\max_{1\leq i\leq k}\{\chi^{z-1}(t_i,t_{i+1})\}+\|\mu_{s,t}^{\tilde{\pi}}-\psi_{s,t}\|_{L(X,Y)},
	\end{align*}
	 for any refinement $\tilde{\pi}$ of $\pi$. Passing to the limit $|\tilde{\pi}|\rightarrow 0$, we obtain 
	 	\[
	 	\|\mu_{s,t}^\pi-\psi_{s,t}\|_{L(X,Y)}\leq  2^z\epsilon_1^2(s,t)\epsilon_2^2(s,t)\zeta(z)\chi(s,t)\max_{1\leq i\leq k}\{\chi^{z-1}(t_i,t_{i+1})\}.
	 	\qedhere
	 	\]
\end{proof}

\section*{Acknowledgement} This work has been funded by Deutsche Forschungsgemeinschaft (DFG) through grant CRC 910 ``Control of self-organizing nonlinear systems: Theoretical methods and concepts of application'', Project (A10) ``Control of stochastic mean-field equations with applications to brain networks.''

\bibliographystyle{abbrv}


\end{document}